\documentclass[12pt,a4paper]{amsart}
\usepackage{amsmath,amsthm,amssymb}
\usepackage{color}
\usepackage[top=30truemm,bottom=30truemm,left=25truemm,right=25truemm]{geometry}
\usepackage{ytableau}
\ytableausetup{boxsize=normal}

\usepackage{tikz}
\usetikzlibrary{decorations.pathreplacing}
\usepackage{hyperref}

\newtheorem{thm}{Theorem}[subsection]
\newtheorem{lem}[thm]{Lemma}
\newtheorem{prop}[thm]{Proposition}
\newtheorem{cor}[thm]{Corollary}

\theoremstyle{definition}
\newtheorem{defi}[thm]{Definition}
\newtheorem{rem}[thm]{Remark}

\newtheorem{ex}[thm]{Example}

\author{Hideya Watanabe}
\address{(H. Watanabe) College of Science, Rikkyo University, 3-34-1, Nishi-Ikebukuro, Toshima-ku, Tokyo, 171-8501, Japan}
\email{watanabehideya@gmail.com}
\date{\today}

\title{Berele row-insertion and quantum symmetric pairs}
\subjclass[2020]{Primary 17B37; Secondary 17B10, 05E10}
\keywords{Berele row-insertion, quantum symmetric pair, symplectic tableau, oscillating tableau, Robinson--Schensted--Knuth correspondence, combinatorial representation theory}

\begin{document}
\maketitle

\begin{abstract}
  The Berele row-insertion is a symplectic analogue of the Schensted row-insertion.
  In the present paper, we provide it with a representation theoretical interpretation via the quantum symmetric pairs of type $A\mathrm{II}$.
  As applications, we lift Berele's Robinson--Schensted correspondence and Kobayashi--Matsumura's Robinson--Schensted--Knuth (RSK for short) correspondence to isomorphisms of representations over a quantum symmetric pair coideal subalgebra, and establish the dual RSK correspondence of type $A\mathrm{II}$.
\end{abstract}


\section{Introduction}
\subsection{Row-insertion and the Berele row-insertion}
It has been known for a long time that combinatorics of semistandard tableaux have intimate connections with representation theory of the general linear Lie algebras $\mathfrak{gl}_n$ over $\mathbb{C}$.
For example, the finite-dimensional irreducible representations of $\mathfrak{gl}_n$ are parametrized by the partitions of length at most $n$, and the representation corresponding to a partition $\lambda$ has a distinguished basis parametrized by the semistandard tableaux of shape $\lambda$ with letters in $\{ 1,\dots,n \}$.

The row-insertion, discovered by Schensted, is an algorithm which takes a pair $(T,x)$ consisting of a semistandard tableau $T$ and a positive integer $x$ as input, and outputs a new semistandard tableau $T \leftarrow x$.
From a representation theoretical point of view, this algorithm can be considered to describe the irreducible decompositions of the tensor products of a finite-dimensional irreducible representation and the vector representation of $\mathfrak{gl}_n$.
The row-insertion is a building block for other algorithms related to representation theory such as the Robinson--Schensted (RS for short) correspondence, Robinson--Schensted--Knuth (RSK for short) correspondence, and dual RSK correspondence (see e.g., \cite{Ful97} for details).

In representation theory of the symplectic Lie algebras $\mathfrak{sp}_{2n}$ over $\mathbb{C}$, the symplectic tableaux (a.k.a.\,the King tableaux) parametrize bases of the finite-dimensional irreducible representations \cite{Kin76}.

Berele \cite{Ber86} discovered a symplectic analogue of the row-insertion.
This algorithm takes a pair $(T,x)$ consisting of a symplectic tableau $T$ and a positive integer $x$ as input, and outputs a new symplectic tableau $T \xleftarrow{\mathrm{B}} x$.
As in the $\mathfrak{gl}_n$ case, the Berele row-insertion describes the irreducible decompositions of the tensor products of a finite-dimensional irreducible representation and the vector representation of $\mathfrak{sp}_{2n}$.
In \cite{Ber86}, Berele used this row-insertion to obtain an RS-type correspondence.
Recently, Kobayashi and Matsumura \cite{KoMa25} established an RSK-type correspondence based on the Berele row-insertion.

\subsection{Quantum groups}
Since the birth of the quantum groups (a.k.a.\,the quantized enveloping algebras) and crystals around 1990, various results in representation theory of $\mathfrak{gl}_n$ have been quantized.
In particular, the algorithms mentioned above have turned out to be isomorphisms of crystals and have been lifted to isomorphisms of representations of the quantum group $U_q(\mathfrak{gl}_n)$; see a survey paper \cite{Kwo09b}.

Based on crystal theory, Kashiwara and Nakashima \cite{KaNa94} introduced new classes of tableaux for representation theory of $\mathfrak{sp}_{2n}$ and the special orthogonal Lie algebras $\mathfrak{so}_n$.
In particular, the Kashiwara--Nakashima tableaux of type $C$, i.e., for $\mathfrak{sp}_{2n}$, are different from the symplectic tableaux.
A row-insertion type algorithm for the Kashiwara--Nakashima tableaux of type $C$ was introduced and studied in \cite{Lec02}.
Of course, it differs from the Berele row-insertion.

Thus, the roles of the symplectic tableaux and Berele row-insertion in the context of quantum groups have been missing.

\subsection{Quantum symmetric pairs}
In a recent paper \cite{Wat23f}, it turned out that the symplectic tableaux naturally appear and work well in representation theory of the quantum symmetric pairs $(U_q(\mathfrak{gl}_{2n}), U^\imath(\mathfrak{sp}_{2n}))$ of type $A\mathrm{II}$ (see also \cite{NaSuWa25}).
Here, $U^\imath(\mathfrak{sp}_{2n})$ is a quantization of $\mathfrak{sp}_{2n}$ that is different from the quantum group $U_q(\mathfrak{sp}_{2n})$ of type $C$.
Unlike $U_q(\mathfrak{sp}_{2n})$, the algebra $U^\imath(\mathfrak{sp}_{2n})$ is a subalgebra of $U_q(\mathfrak{gl}_{2n})$.
Therefore, each $U_q(\mathfrak{gl}_{2n})$-module can be regarded as a $U^\imath(\mathfrak{sp}_{2n})$-module via the embedding.
By ``representation theory of $(U_q(\mathfrak{gl}_{2n}),U^\imath(\mathfrak{sp}_{2n}))$'', we mean the study of such $U^\imath(\mathfrak{sp}_{2n})$-modules.
For a general theory of quantum symmetric pairs, we refer the reader to \cite{Kol14}.

Let $\mathrm{Par}_{\leq n}$ denote the set of partitions of length at most $n$.
The finite-dimensional irreducible representations of $\mathfrak{sp}_{2n}$ are parametrized by $\mathrm{Par}_{\leq n}$.
Each such representation can be quantized to a finite-dimensional irreducible representation of $U^\imath(\mathfrak{sp}_{2n})$ \cite{Mol06}, \cite{Wat21b}.
Let $V^\imath(\nu)$ denote the corresponding representation.
It was proved in \cite{Wat23f} that $V^\imath(\nu)$ has a distinguished basis of the form
\[
  \{ b^\imath_T \mid T \in Sp\mathrm{T}_{2n}(\nu) \},
\]
where $Sp\mathrm{T}_{2n}(\nu)$ denotes the set of symplectic tableaux of shape $\nu$.
This basis has several good properties in common with the canonical basis of a finite-dimensional irreducible representation $V(\lambda)$ of $U_q(\mathfrak{gl}_{2n})$; $\lambda \in \mathrm{Par}_{\leq 2n}$ represents the highest weight of the representation.

In \cite{Wat23f}, an algorithm that transforms each semistandard tableau $T$ into a symplectic tableau $P(T)$ was introduced.
It describes the irreducible decompositions of the finite-dimensional irreducible representation of $U_q(\mathfrak{gl}_{2n})$ as a $U^\imath(\mathfrak{sp}_{2n})$-module.
Hence, one can naturally define an analogue of the row-insertion by
\[
  T \xleftarrow{A\mathrm{II}} x := P(T \leftarrow x).
\]
We call it the row-insertion of type $A\mathrm{II}$.

\subsection{Results}
The main result in the present paper is that the row-insertion of type $A\mathrm{II}$ coincides with the Berele row-insertion (Theorem \ref{thm: row ins AII = ber row ins}):
\[
  T \xleftarrow{A\mathrm{II}} x = T \xleftarrow{\mathrm{B}} x.
\]
This suggests that the symplectic tableaux and the Berele row-insertion should be thought of as a type $A\mathrm{II}$ analogue of the semistandard tableaux and the Schensted row-insertion, but not type $C$.
We prove this result by rewriting the two algorithms in terms of the Schensted row-insertion and the Sch\"{u}tzenberger sliding.

As applications, we lift the Berele RS correspondence and the Kobayashi--Matsumura RSK correspondence to $U^\imath(\mathfrak{sp}_{2n})$-isomorphisms (Theorems \ref{thm: rs aii} and \ref{thm: rsk aii}).
Namely, there exist $\mathbf{U}^\imath$-module isomorphisms
\[
  \mathrm{RS}^{A\mathrm{II}} : V^\imath(\nu) \otimes V(1)^{\otimes N} \to \bigoplus_{\xi \in \mathrm{Par}_{\leq n}} (V^\imath(\xi) \otimes \mathbb{Q}(q) \mathrm{OT}_{n,N}(\nu,\xi))
\]
and
\begin{align*}
  \mathrm{RSK}^{A\mathrm{II}} : &V^\imath(\nu) \otimes V(l_1) \otimes \cdots \otimes V(l_k) \\
  &\to \bigoplus_{\xi \in \mathrm{Par}_{\leq n}} V^\imath(\xi) \otimes \mathbb{Q}(q) \{ U \in \mathrm{CSOT}_{n,k}(\nu,\xi) \mid c(U) = (l_1,\dots,l_k) \}
\end{align*}
which recovers the Berele RS correspondence and the Kobayashi--Matsumura RSK correspondence at $q = \infty$, respectively.
Here, $\mathrm{OT}_{n,k}(\nu,\xi)$ denotes the set of \emph{oscillating tableaux} (a.k.a.\,\emph{up-down tableaux}) (Definition \ref{def: ot}) and $\mathrm{CSOT}_{n,k}(\nu,\xi)$ the set of \emph{column-strict oscillating tableaux} (Definition \ref{def: csot}).

Moreover, we establish the dual RSK correspondence of type $A\mathrm{II}$ and its quantum lift (Theorem \ref{thm: dual rsk aii}):
\begin{align*}
  \mathrm{dRSK}^{A\mathrm{II}}: &V^\imath(\nu) \otimes V(1^{k_1}) \otimes \cdots \otimes V(1^{k_l}) \\
  &\to \bigoplus_{\xi \in \mathrm{Par}_{\leq n}} V^\imath(\xi) \otimes \mathbb{Q}(q) \{ U \in \mathrm{RSOT}_{n,l}(\nu,\xi) \mid c(U) = (k_1,\dots,k_l) \}.
\end{align*}
Here, $\mathrm{RSOT}_{n,k}(\nu,\xi)$ denotes the set of \emph{row-strict oscillating tableaux} (Definition \ref{def: rsot}).

\subsection{Organization}
This paper is organized as follows.
In Section \ref{sect: par}, we prepare necessary notions regarding partitions such as horizontal strips, vertical strips, and punctured partitions.
Section \ref{sect: tab} is devoted to reviewing the row-insertion and the sliding on semistandard tableaux.
In Section \ref{sect: ber row-ins}, we study the Berele row-insertion, the central research object in the present paper.
We briefly review representation theory of the quantum symmetric pairs of type $A\mathrm{II}$ in Section \ref{sect: row-ins aii}.
We also prove the coincidence of the Berele row-insertion and the row-insertion of type $A\mathrm{II}$ there.
In Section \ref{sect: app}, we use this result to deduce various combinatorial and representation theoretical results mentioned above.

\subsection{Acknowledgments}
The author would like to thank the Research Institute for Mathematical Sciences, an International Joint Usage/Research Center located in Kyoto University, where he got the initial idea of this work during his stay in June 2025.
This work was supported by JSPS KAKENHI Grant Number JP24K16903.

\subsection{Notation}
Throughout this paper, we fix a positive integer $n \in \mathbb{Z}_{> 0}$.

For each nonnegative integers $a,b \in \mathbb{Z}_{\geq 0}$, we set
\begin{align*}
  &[a,b] := \{ c \in \mathbb{Z} \mid a \leq c \leq b \}, \\
  &[a] := [1,a].
\end{align*}

Let $\leq_\mathrm{lex}$ denote the lexicographic order on $\mathbb{Z}^2$:
we have $(i,j) \leq_\mathrm{lex} (k,l)$ if and only if either $i < k$, or $i = k$ and $j \leq l$.

\section{Partitions}\label{sect: par}
Partitions and skew partitions naturally appear in representation theory of $\mathfrak{gl}_n$.
Besides them, partitions with ``holes'' often appear when we manipulate tableaux.
For rigorous treatment of such tableaux, we introduce the notion of punctured partitions.

\subsection{Partitions}
A \emph{partition} is a weakly decreasing finite sequence $\lambda = (\lambda_1,\dots,\lambda_l)$ of positive integers.
The integers $\lambda_1,\dots,\lambda_l$ are called the \emph{parts} of $\lambda$.
The length $l$ is referred to as the \emph{length} of $\lambda$, and denoted by $\ell(\lambda)$.
The sum $|\lambda| := \sum_{i=1}^l \lambda_i$ of parts is called the \emph{size} of $\lambda$.

Let $\mathrm{Par}$ denote the set of partitions.
Also, for each $l \in \mathbb{Z}_{\geq 0}$, let $\mathrm{Par}_{\leq l}$ denote the partitions of length at most $l$.

Given $\lambda \in \mathrm{Par}$, it is often convenient to extend the notion of parts by setting $\lambda_0 := \infty$ and $\lambda_i := 0 \text{ for all } i > \ell(\lambda)$:
\[
  \lambda = (\infty,\lambda_1,\dots,\lambda_{\ell(\lambda)},0,0,\dots)
\]

The \emph{Young diagram} of $\lambda \in \mathrm{Par}$ is the set
\[
  D(\lambda) := \{ (i,j) \in \mathbb{Z}_{> 0}^2 \mid j \in [\lambda_i] \}.
\]
We represent it by a collection of boxes arranged in left-justified rows with the $i$-th row from the top consisting of $\lambda_i$ boxes.
For example, if $\lambda = (5,4,3,3,1)$, then
\[
  D(\lambda) = \ydiagram{5,4,3,3,1}.
\]

\begin{defi}\label{def: add rm par}
  Let $\lambda \in \mathrm{Par}$ and $r \in \mathbb{Z}_{> 0}$.
  \begin{enumerate}
    \item The integer $r$ is said to be an \emph{addable row} of $\lambda$ if
    \[
      \lambda_r < \lambda_{r-1}.
    \]
    In this case, we set
    \[
      \operatorname{add}(\lambda,r) := (\dots,\lambda_{r-1},\lambda_r+1,\lambda_{r+1},\dots) \in \mathrm{Par}
    \]
    \item The integer $r$ is said to be a \emph{removable row} of $\lambda$ if
    \[
      \lambda_r > \lambda_{r+1}.
    \]
    In this case, we set
    \[
      \operatorname{rm}(\lambda,r) := (\dots,\lambda_{r-1},\lambda_r-1,\lambda_{r+1},\dots) \in \mathrm{Par}.
    \]
  \end{enumerate}
\end{defi}

\begin{ex}
  Let $\lambda = (5,4,3,3,1)$.
  The addable rows of $\lambda$ are $1,2,3,5,6$, and the removable rows are $1,2,4,5$.
  We have
  \[
    D(\operatorname{add}(\lambda,6)) = \ydiagram{5,4,3,3,1,1}, \quad D(\operatorname{rm}(\lambda,2)) = \ydiagram{5,3,3,3,1}.
  \]
\end{ex}

\begin{lem}\label{lem: rm add}
  Let $\lambda \in \mathrm{Par}$ and $r \in \mathbb{Z}_{> 0}$.
  \begin{enumerate}
    \item If $r$ is an addable row of $\lambda$, then it is a removable row of $\operatorname{add}(\lambda,r)$ and we have
    \[
      \operatorname{rm}(\operatorname{add}(\lambda,r),r) = \lambda.
    \]
    \item If $r$ is a removable row of $\lambda$, then it is an addable row of $\operatorname{rm}(\lambda,r)$ and we have
    \[
      \operatorname{add}(\operatorname{rm}(\lambda,r),r) = \lambda.
    \]
  \end{enumerate}
\end{lem}
\begin{proof}
  The assertions are immediate from definitions.
\end{proof}

\subsection{Skew partitions}
Define a binary relation $\subseteq$ on the set $\mathrm{Par}$ of partitions by declaring $\mu \subseteq \lambda$ to mean that $D(\mu) \subseteq D(\lambda)$, or equivalently,
\[
  \mu_i \leq \lambda_i \quad \text{ for all } i \geq 1.
\]

A \emph{skew partition} is a pair $\lambda/\mu := (\lambda,\mu)$ of partitions such that $\mu \subseteq \lambda$.
The \emph{size} of $\lambda/\mu$ is the difference $|\lambda/\mu| := |\lambda|-|\mu|$.
The \emph{Young diagram} of $\lambda/\mu$ is the set difference $D(\lambda/\mu) := D(\lambda) \setminus D(\mu)$.

\begin{ex}
  Let $\lambda := (5,4,3,3,1)$ and $\mu := (4,3,1)$.
  Then, we have $\mu \subset \lambda$, $|\lambda/\mu| = 8$, and
  \[
    D(\lambda/\mu) = \ydiagram{4+1,3+1,1+2,3,1}.
  \]
\end{ex}

We often identify a partition $\lambda$ with the skew partition $\lambda/()$, where $()$ is the unique partition of length $0$.

\begin{defi}
  Let $\lambda/\mu$ be a skew partition.
  \begin{enumerate}
    \item We say that $\lambda/\mu$ is a \emph{horizontal strip} if its Young diagram has at most one box in each column:
    \[
      \mu_i \in [\lambda_{i+1},\lambda_i] \quad \text{ for all } i \geq 1.
    \]
    We write $\mu \overset{\text{hor}}{\subseteq} \lambda$ to mean that $\lambda/\mu$ is a horizontal strip.
    \item We say that $\lambda/\mu$ is a \emph{vertical strip} if its Young diagram has at most one box in each row:
    \[
      \mu_i \in \{ \lambda_i-1, \lambda_i \} \quad \text{ for all } i \geq 1.
    \]
    We write $\mu \overset{\text{ver}}{\subseteq} \lambda$ to mean that $\lambda/\mu$ is a vertical strip.
  \end{enumerate}  
\end{defi}

\begin{ex}
  Let $\lambda := (5,4,3,3,1)$, $\mu := (4,4,3,1,1)$, $\nu := (5,3,2,2)$.
  Then, $\lambda/\mu$ is a horizontal strip, and $\lambda/\nu$ is a vertical strip:
  \[
    D(\lambda/\mu) = \ydiagram{4+1,4+0,3+0,1+2}, \quad D(\lambda/\nu) = \ydiagram{5+0,3+1,2+1,2+1,1}.
  \]
\end{ex}

\subsection{Punctured partitions}
\begin{defi}
  A \emph{punctured partition} is a pair $(\lambda,H)$ consisting of a partition $\lambda$ and a subset $H$ of the Young diagram $D(\lambda)$ of $\lambda$.
  The elements of $H$ are called the \emph{holes} of $(\lambda,H)$.
  The \emph{Young diagram} of a punctured partition $(\lambda,H)$ is the set difference
  \[
    D(\lambda,H) := D(\lambda) \setminus H.
  \]
\end{defi}

We represent the Young diagram of a punctured partition $(\lambda,H)$ by filling in each box of $D(\lambda)$ corresponding to a hole with a circle.
When we do not care whether a box in $D(\lambda)$ is a hole or not, we paint it gray.

\begin{ex}\label{ex: punc par}
  Let $\lambda = (5,4,3,3,1)$ and $H := \{ (1,2), (1,3), (2,4), (3,1), (4,1), (4,2) \}$.
  Then, $(\lambda,H)$ is a punctured partition, and we have
  \[
    D(\lambda,H) = \begin{ytableau}
    \empty & \bigcirc & \bigcirc & \empty & \empty \\
    \empty & \empty & \empty & \bigcirc \\
    \bigcirc & \empty & \empty \\
    \bigcirc & \bigcirc & \empty \\
    \empty \\
    \end{ytableau}.
  \]
\end{ex}

\begin{rem}\label{rem: skew par is punc par}
  Let $\lambda/\mu$ be a skew partition.
  Then, $(\lambda,D(\mu))$ is a punctured partition.
  We often identify them.
  In particular, $\lambda = \lambda/()$ is a punctured partition without holes.
\end{rem}

\begin{defi}
  Let $(\lambda,H)$ be a punctured partition, and $(r,c) \in H$.
  \begin{enumerate}
    \item We say that $(r,c)$ is \emph{slidable} if either $(r,c+1)$ or $(r+1,c)$ is a member of $D(\lambda,H)$.
    \item We say that $(r,c)$ is \emph{reversely slidable} if either $(r,c-1)$ or $(r-1,c)$ is a member of $D(\lambda,H)$.
  \end{enumerate}
\end{defi}

\begin{ex}
  Let $(\lambda,H)$ be as in Example \ref{ex: punc par}.
  Then, the slidable holes are $(1,2)$, $(1,3)$, $(3,1)$, $(4,1)$, $(4,2)$, while the reversely slidable holes are $(1,2)$, $(2,4)$, $(3,1)$, $(4,2)$.
\end{ex}

\begin{lem}\label{lem: rows of punc par}
  Let $(\lambda,H)$ be a punctured partition, and $(r,c) \in \mathbb{Z}_{\geq 0}^2$.
  \begin{enumerate}
    \item\label{item: sl lem: rows of punc par} Suppose that there is no slidable hole of the form $(r,c')$ with $c' > c$, and set
    \[
      q := \min\{ j \geq 0 \mid (r,c+j+1) \notin D(\lambda,H) \}.
    \]
    Then, we have
    \[
      \{ j > c \mid (r,j) \in D(\lambda,H) \} = [c+1,c+q];
    \]
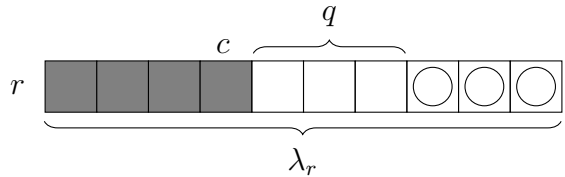
    \[
      \begin{tikzpicture}
        \draw[fill=gray, scale=1.7]  (-1.2,0) rectangle (-0.8,-0.4);
        \draw[fill=gray, scale=1.7]  (-0.8,0) rectangle (-0.4,-0.4);
        \draw[fill=gray, scale=1.7]  (-0.4,0) rectangle (0,-0.4);
        \draw[fill=gray, scale=1.7] (0,0) rectangle (0.4,-0.4);
        \draw[scale=1.7] (0.4,0) rectangle (0.8,-0.4);
        \draw[scale=1.7] (0.8,0) rectangle (1.2,-0.4);
        \draw[scale=1.7] (1.2,0) rectangle (1.6,-0.4);
        \draw[scale=1.7] (1.6,0) rectangle (2,-0.4);
        \draw[scale=1.7] (2,0) rectangle (2.4,-0.4);
        \draw[scale=1.7] (2.4,0) rectangle (2.8,-0.4);
        \draw[scale=1.7] (1.8,-0.2) circle (0.15);
        \draw[scale=1.7] (2.2,-0.2) circle (0.15);
        \draw[scale=1.7] (2.6,-0.2) circle (0.15);
        \draw [decorate, decoration={brace, amplitude=5pt}](0.7,0.1)  --  node[above,yshift=0.5em]{$q$} (2.7,0.1);
        \draw [decorate, decoration={brace, amplitude=5pt, mirror}](-2.05,-0.8) --  node[below,yshift=-0.5em]{$\lambda_r$} (4.75,-0.8);
        \node at (0.3,0.2) {$c$};
        \node at (-2.4,-0.35) {$r$};
      \end{tikzpicture}
    \]
    \item\label{item: rev sl lem: rows of punc par} Suppose that there is no reversely slidable hole of the form $(r,c')$ with $c' < c$, and set
    \[
      p := \min\{ j \geq 0 \mid (r,c-j-1) \notin D(\lambda,H) \}.
    \]
    Then, we have
    \[
      \{ j < c \mid (r,j) \in D(\lambda,H) \} = [c-p,c-1];
    \]
    
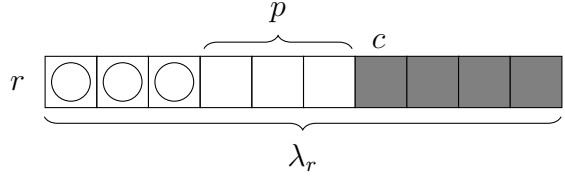
\[
      \begin{tikzpicture}
        \draw[scale=1.7] (-1.2,0) rectangle (-0.8,-0.4);
        \draw[scale=1.7] (-0.8,0) rectangle (-0.4,-0.4);
        \draw[scale=1.7] (-0.4,0) rectangle (0,-0.4);
        \draw[scale=1.7] (0,0) rectangle (0.4,-0.4);
        \draw[scale=1.7] (0.4,0) rectangle (0.8,-0.4);
        \draw[scale=1.7] (0.8,0) rectangle (1.2,-0.4);
        \draw[fill=gray, scale=1.7] (1.2,0) rectangle (1.6,-0.4);
        \draw[fill=gray, scale=1.7] (1.6,0) rectangle (2,-0.4);
        \draw[fill=gray, scale=1.7] (2,0) rectangle (2.4,-0.4);
        \draw[fill=gray, scale=1.7] (2.4,0) rectangle (2.8,-0.4);
        \draw[scale=1.7] (-1,-0.2) circle (0.15);
        \draw[scale=1.7] (-0.6,-0.2) circle (0.15);
        \draw[scale=1.7] (-0.2,-0.2) circle (0.15);
        \draw [decorate, decoration={brace, amplitude=5pt}](0.05,0.1)  --  node[above,yshift=0.5em]{$p$} (2,0.1);
        \draw [decorate, decoration={brace, amplitude=5pt, mirror}](-2.05,-0.8) --  node[below,yshift=-0.5em]{$\lambda_r$} (4.75,-0.8);
        \node at (2.35,0.2) {$c$};
        \node at (-2.4,-0.35) {$r$};
      \end{tikzpicture}
    \]
  \end{enumerate}
\end{lem}
\begin{proof}
  We prove only the first assertion; the second one can be proved similarly.
  By the definition of $q$, we have
  \[
    (r,j) \in D(\lambda,H) \quad \text{ for all } j \in [c+1,c+q].
  \]
  Hence, we only need to show that $(r,j) \notin D(\lambda,H)$ for all $j > c+q$.
  Assume contrary that $(r,j) \in D(\lambda,H)$ for some $j > c+q$.
  We may further assume that this $j$ is the minimum among those.
  Since $(r,c+q+1) \notin D(\lambda,H)$ by the definition of $q$, we must have $j > c+q+1$ and $(r,j-1) \notin D(\lambda,H)$.
  This implies that $(r,j-1)$ is a slidable hole.
  However, this contradicts our assumption on $(r,c)$.
  Thus, we complete the proof.
\end{proof}

\begin{prop}\label{prop: punc par wo (rev)slidable holes}
  Let $(\lambda,H)$ be a punctured partition.
  \begin{enumerate}
    \item\label{item: 1 lem punc par wo (rev)slidable holes} If $(\lambda,H)$ has no slidable holes, then there exists a partition $\mu \subseteq \lambda$ such that $D(\lambda,H) = D(\mu)$.
    \item\label{item: 2 lem punc par wo (rev)slidable holes} If $(\lambda,H)$ has no reversely slidable holes, then there exists a partition $\nu \subseteq \lambda$ such that $D(\lambda,H) = D(\lambda/\nu)$.
  \end{enumerate}
\end{prop}
\begin{proof}
  Let us prove the first assertion.
  For each $i \geq 1$, set
  \[
    \mu_i := \min\{ j \geq 0 \mid (i,j+1) \notin D(\lambda,H) \}.
  \]
  By Lemma \ref{lem: rows of punc par} \eqref{item: sl lem: rows of punc par}, we have
  \[
    D(\lambda,H) = \{ (i,j) \in \mathbb{Z}_{> 0}^2 \mid j \in [\mu_i] \}.
  \]

  It remains to show that $\mu := (\mu_1,\mu_2,\dots)$ is a partition contained in $\lambda$.
  By the definition of $\mu_i$'s, we see that $\mu_i \leq \lambda_i$ for all $i \geq 1$.
  In particular, $\mu_i = 0$ for all $i > \ell(\lambda)$.
  Hence, we only need to prove that $\mu_i \geq \mu_{i+1}$ for all $i \geq 1$.
  Assume contrary that $\mu_i < \mu_{i+1}$ for some $i \geq 1$.
  Then, we have $(i,\mu_{i+1}) \notin D(\lambda,H)$ and $(i+1,\mu_{i+1}) \in D(\lambda,H)$.
  Since $\mu_{i+1} \leq \lambda_{i+1} \leq \lambda_i$, these imply that $(i,\mu_{i+1})$ is a slidable hole.
  However, this contradicts our assumption that $(\lambda,H)$ has no slidable holes.
  Thus, we complete the proof of the first assertion.

  The second assertion can be prove in a similar way to the first one; we use Lemma \ref{lem: rows of punc par} \eqref{item: rev sl lem: rows of punc par} instead of \eqref{item: sl lem: rows of punc par} in this case.
\end{proof}

\begin{rem}\label{rem: punc par wo (rev)slidable holes}
  Based on Proposition \ref{prop: punc par wo (rev)slidable holes}, we often regard punctured partitions without slidable holes as partitions, and those without reversely slidable holes as skew partitions.
\end{rem}

\section{Semistandard tableaux}\label{sect: tab}
In this section, we study various algorithms involving semistandard tableaux such as the Schensted row-insertion, the Sch\"{u}tzenberger sliding, and the rectification.
They are used to define the Berele row-insertion and the row-insertion of type $A\mathrm{II}$ in the subsequent sections.

Throughout this section, we fix a punctured partition $(\lambda,H)$.

\subsection{Semistandard punctured tableaux}
\begin{defi}\label{def: punc tab}
  A \emph{punctured tableau} of shape $(\lambda,H)$ with entries in $[n]$ is a map
  \[
    T : D(\lambda,H) \to [n].
  \]
  Let $\operatorname{sh}(T) := (\lambda,H)$ denote the shape of $T$.
  The values $T(i,j)$ with $(i,j) \in D(\lambda,H)$ are called the \emph{entries} of $T$.
\end{defi}

Given a punctured tableau $T$ of shape $(\lambda,H)$, we set
\[
  T(i,j) := \infty \quad \text{ for each } (i,j) \in \mathbb{Z}_{> 0}^2 \setminus D(\lambda,H).
\]

A tableau $T$ of shape $(\lambda,H)$ is represented by filling in the boxes of the Young diagram $D(\lambda,H)$ with their entries.

\begin{ex}\label{ex: punc tab}
  Let $(\lambda,H)$ as in Example \ref{ex: punc par}.
  The following is a tableau of shape $(\lambda,H)$.
  \[
    \begin{ytableau}
      2 & \bigcirc & \bigcirc & 2 & 4 \\
      3 & 4 & 4 & \bigcirc \\
      \bigcirc & 6 & 8 \\
      \bigcirc & \bigcirc & 9 \\
      5
    \end{ytableau}
  \]
\end{ex}

\begin{rem}
  The notion of punctured tableaux has already appeared in \cite{Ber86}, but in a more restricted meaning.
  Namely, the punctured tableaux there are always supposed to have exactly one hole; $|H| = 1$.
\end{rem}

  When we regard a punctured partition $(\lambda,H)$ without slidable holes as a partition $\mu$ as in Remark \ref{rem: punc par wo (rev)slidable holes}, we say that a punctured tableau $T$ of shape $(\lambda,H)$ is a \emph{tableau} of shape $\mu$ and write $\operatorname{sh}(T) = \mu$.
  Similarly, when we regard a punctured partition $(\lambda,H)$ without reversely slidable holes as a skew partition $\lambda/\nu$, we say that a punctured tableau $T$ of shape $(\lambda,H)$ is a \emph{skew tableau} of shape $\lambda/\nu$ and write $\operatorname{sh}(T) = \lambda/\nu$.

\begin{defi}\label{def: sstd punc tab}
  A punctured tableau $T$ of shape $(\lambda,H)$ with entries in $[n]$ is said to be \emph{semistandard} if its entries weakly increase along the rows from the left to right and strictly increase along the columns from the top to bottom: for each $(i,j),(k,l) \in D(\lambda,H)$, it holds that
  \begin{enumerate}
    \item $T(i,j) \leq T(k,l)$ if $i = k$ and $j < l$,
    \item $T(i,j) < T(k,l)$ if $i < k$ and $j = l$.
  \end{enumerate}
\end{defi}

\begin{ex}
  The punctured tableau in Example \ref{ex: punc tab} is semistandard.
\end{ex}

Let $\mathrm{SST}_n(\lambda,H)$ denote the set of semistandard punctured tableaux of shape $(\lambda,H)$ with entries in $[n]$.
Similarly, let $\mathrm{SST}_n(\mu)$ denote the set of semistandard tableaux of shape $\mu$ with entries in $[n]$, and $\mathrm{SST}_n(\lambda/\nu)$ the set of semistandard skew tableaux of shape $\lambda/\nu$ with entries in $[n]$.

\subsection{Words}
Let $\mathcal{W}_n$ denote the free monoid of words with letters in $[n]$:
\[
  \mathcal{W}_n := \bigsqcup_{l \geq 0} \,[n]^l.
\]

\begin{defi}\label{def: knu equiv}
  Two words $w,w' \in \mathcal{W}_n$ are said to be \emph{Knuth equivalent} if they can be transformed into each other by a sequence of the elementary Knuth transformations
  \begin{description}
    \item[(K1)] $yzx \mapsto yxz$ if $x < y \leq z$,
    \item[(K2)] $xzy \mapsto zxy$ if $x \leq y < z$,
  \end{description}
  and their inverses.
  In this case, we write $w \equiv w'$.
\end{defi}

\begin{defi}
Let $T$ be a punctured tableau of shape $(\lambda,H)$.
  \begin{enumerate}
    \item The \emph{row-word} of $T$ is the word $w_\mathrm{row}(T)$ obtained by reading the entries of $T$ along the rows from the left to right and from the bottom to top.
    \item The \emph{column-word} of $T$ is the word $w_\mathrm{col}(T)$ obtained by reading the entries of $T$ along the columns from the bottom to top and from the left to right.
  \end{enumerate}
\end{defi}

\begin{ex}
  Let $T$ denote the punctured tableau in Example \ref{ex: punc tab}.
  Then, we have
  \[
    w_\mathrm{row}(T) = (5,9,6,8,3,4,4,2,2,4), \quad w_\mathrm{col}(T) = (5,3,2,6,4,9,8,4,2,4).
  \]
\end{ex}

Regarding the Knuth equivalence, let us recall the following results for later use.

\begin{prop}[\textit{cf}.\,{\cite[\S 2.1 Theorem]{Ful97}}]\label{prop: row word equiv}
  Let $T,S$ be semistandard tableaux.
  If $w_\mathrm{row}(T) \equiv w_\mathrm{row}(S)$, then we have $T = S$.
\end{prop}

\begin{prop}[{\cite[\S 2.3 equation (10)]{Ful97}}]\label{prop: row col}
  Let $T$ be a semistandard skew tableau.
  Then, we have
  \[
    w_\mathrm{row}(T) \equiv w_\mathrm{col}(T).
  \]
\end{prop}

\begin{lem}[{\cite[the proof of \S 2.1 Proposition 2]{Ful97}}]\label{lem: sl lem}
  Let $p,q \in \mathbb{Z}_{\geq 0}$ and $u_1,\dots,u_p$, $v_1,\dots,v_p$, $x$, $y_1,\dots,y_q$, $z_1,\dots,z_q \in \mathbb{Z}_{> 0}$.
  Assume the following\,\rm{:}
  \begin{itemize}
    \item $u_1 \leq \cdots \leq u_p$,
    \item $v_1 \leq \cdots \leq v_p$,
    \item $u_i < v_i$ for all $i \in [p]$,
    \item $v_p \leq x \leq y_1$,
    \item $y_1 \leq \cdots \leq y_q$,
    \item $z_1 \leq \cdots \leq z_q$,
    \item $y_j < z_j$ for all $j \in [q]$.
  \end{itemize}
  Then, we have
  \begin{align}
    v_1 \cdots v_p x z_1 \cdots z_q u_1 \cdots u_p y_1 \cdots y_q \equiv v_1 \cdots v_p z_1 \cdots z_q u_1 \cdots u_p x y_1 \cdots y_q. \label{eq: sl lem}
  \end{align}
\end{lem}

\begin{rem}\label{rem: sl lem}
  The left-hand side of equation \eqref{eq: sl lem} in Lemma \ref{lem: sl lem} equals the row-word of the punctured tableau
  \[
    \begin{ytableau}
      u_1 & \cdots & u_p & \bigcirc & y_1 & \cdots & y_q \\
      v_1 & \cdots & v_p & x & z_1 & \cdots & z_q
    \end{ytableau},
  \]
  while the right-hand side
  \[
    \begin{ytableau}
      u_1 & \cdots & u_p & x & y_1 & \cdots & y_q \\
      v_1 & \cdots & v_p & \bigcirc & z_1 & \cdots & z_q
    \end{ytableau}.
  \]
  This observation would help us to relate the Knuth equivalence and the sliding algorithm.
\end{rem}

\subsection{Row-insertion}
\begin{defi}\label{def: add(T,r,x)}
  Let $T \in \mathrm{SST}_n(\lambda)$, $r$ an addable row of $\lambda$, and $x \in [n]$.
  Define a tableau $\operatorname{add}(T,r,x)$ of shape $\operatorname{add}(\lambda,r)$ by
  \[
    (\operatorname{add}(T,r,x))(i,j) := \begin{cases}
      x & \text{ if } (i,j) = (r,\lambda_r+1), \\
      T(i,j) & \text{ otherwise},
    \end{cases}
  \]
  for each $(i,j) \in D(\operatorname{add}(\lambda,r))$.
\end{defi}

\begin{defi}\label{def: row-ins}
  The \emph{row-insertion} is the algorithm which takes $(T,x) \in \mathrm{SST}_n(\lambda) \times [n]$ as an input, and outputs a new semistandard tableau $T \leftarrow x$ as follows.
  \begin{enumerate}
    \item Initialize $r := 1$, $T_r := T$, $x_r := x$.
    \item\label{step: 2 row-ins} Set $c_r := \min\{ j \geq 1 \mid T(r,j) > x_r \}$.
    \item \begin{enumerate}
      \item If $c_r = \lambda_r+1$, then return $\operatorname{add}(T,r,x)$.
      \item Otherwise, set $T_{r+1}$ to be the tableau of shape $\lambda$ obtained from $T_r$ by replacing the $(r,c_r)$-entry with $x_r$:
      \[
        T_{r+1}(i,j) := \begin{cases}
          x_r & \text{ if } (i,j) = (r,c_r), \\
          T_r(i,j) & \text{ otherwise},
        \end{cases}
      \]
      for each $(i,j) \in D(\lambda)$.
      Also, set $x_{r+1} := T_r(r,c_r)$, increment $r$, and then go back to step \eqref{step: 2 row-ins}.
    \end{enumerate}
  \end{enumerate}
  As a byproduct of this algorithm, we obtain the sequences $(c_1,\dots,c_r)$ and $(x_1,\dots,x_r)$.
  We call them the \emph{row-insertion route} and the \emph{row-inserting letters} for $(T,x)$, respectively.
\end{defi}

\begin{ex}
  Let
  \[
    T := \begin{ytableau}
      1 & 1 & 2 & 3 & 3 \\
      3 & 3 & 4 & 8 \\
      6 & 6 & 8 \\
      8 & 8 & 9 \\
      9
    \end{ytableau}.
  \]
  Then,
  \[
    T \leftarrow 2 = \begin{ytableau}
      1 & 1 & 2 & *(lightgray) 2 & 3 \\
      3 & 3 & *(lightgray) 3 & 8 \\
      *(lightgray) 4 & 6 & 8 \\
      *(lightgray) 6 & 8 & 9 \\
      *(lightgray) 8 \\
      *(lightgray) 9
    \end{ytableau},
  \]
  where the shaded boxes represent the row-insertion route and the row-inserting letters; they are $(4,3,1,1,1,1)$ and $(2,3,4,6,8,9)$, respectively.
\end{ex}

\begin{prop}\label{prop: explicit desc row ins}
  Let $T \in \mathrm{SST}_n(\lambda)$ and $x \in [n]$.
  Set $S := T \leftarrow x$, $\mu := \operatorname{sh}(S)$, and let $(c_1,\dots,c_r)$ and $(x_1,\dots,x_r)$ denote the row-insertion route and the row-inserting letters for $(T,x)$, respectively.
  Then, the following hold:
  \begin{enumerate}
    \item\label{item: 1 prop: explicit desc row ins} $r \in [\ell(\lambda)+1]$,
    \item\label{item: 2 prop: explicit desc row ins} $c_i = \min\{ j \geq 1 \mid T(i,j) > x_i \}$ for all $i \in [r]$,
    \item\label{item: 3 prop: explicit desc row ins} $c_1 \geq \cdots \geq c_r = \lambda_r+1 = \mu_r$,
    \item\label{item: 4 prop: explicit desc row ins} $x_1 = x$ and $x_i = T(i-1,c_{i-1})$ for all $i \in [2,r]$,
    \item\label{item: 5 prop: explicit desc row ins} $x_1 < \cdots < x_r$,
    \item\label{item: 6 prop: explicit desc row ins} $\mu = \operatorname{add}(\lambda,r)$,
    \item\label{item: 7 prop: explicit desc row ins} for each $(i,j) \in D(\mu)$, we have
    \[
      S(i,j) = \begin{cases}
        x_i & \text{ if } i \in [r] \text{ and } j = c_i, \\
        T(i,j) & \text{ otherwise}.
      \end{cases}
    \]
  \end{enumerate}
\end{prop}
\begin{proof}
  The assertions are straightforwardly deduced from Definition \ref{def: row-ins}.
\end{proof}

\begin{thm}[{\cite[\S 1.1]{Ful97}}]
  The row-insertion gives rise to a bijection
  \[
    \mathrm{SST}_n(\lambda) \times [n] \to \bigsqcup_{\substack{\mu \in \mathrm{Par}_{\leq n} \\ \mu \supset \lambda,\ |\mu/\lambda|=1}} \mathrm{SST}_n(\mu).
  \]
\end{thm}

\subsection{Sliding and the rectification}\label{subsect: sl rect}
\begin{defi}[\textit{cf}.\,{\cite[\S 1.2]{Ful97}}]\label{def: sl}
  The \emph{sliding} is the algorithm which takes a pair $(T,(r,c))$ consisting of a semistandard punctured tableau $T$ of shape $(\lambda,H)$ and a hole $(r,c) \in H$ as input, and outputs a new semistandard punctured tableau $\operatorname{Sl}(T;r,c)$ as follows.
  \begin{enumerate}
    \item Initialize $t := 0$, $T_t := T$, $r_t := r$, $c_t := c$, $H_t := H$.
    \item\label{step: 2 sl} \begin{enumerate}
      \item If $(r_t,c_t)$ is not slidable, then return $T_t$.
      \item If $T(r_t,c_t+1) < T(r_t+1,c_t)$, then set $r_{t+1} := r_t$, $c_{t+1} := c_t+1$.
      \item If $T(r_t,c_t+1) \geq T(r_t+1,c_t)$, then set $r_{t+1} := r_t+1$, $c_{t+1} := c_t$.
    \end{enumerate}
    \item Set $H_{t+1} := (H_t \setminus \{(r_t,c_t)\}) \sqcup \{(r_{t+1},c_{t+1})\}$, and $T_{t+1}$ to be the semistandard punctured tableau of shape $(\lambda,H_{t+1})$ defined by
    \[
      T_{t+1}(i,j) := \begin{cases}
        T_t(r_{t+1},c_{t+1}) & \text{ if } (i,j) = (r_t,c_t), \\
        T_t(i,j) & \text{ otherwise},
      \end{cases}
    \]
    for each $(i,j) \in D(\lambda,H_{t+1})$.
    \item Increment $t$ and then go back to step \eqref{step: 2 sl}.
  \end{enumerate}
\end{defi}

\begin{ex}\label{ex: sl}
  Let
  \[
    T := \begin{ytableau}
      1 & 1 & 2 & 2 & 3 \\
      \bigcirc & 3 & 3 & 8 \\
      6 & 6 & 8 \\
      8 & 8 & 9 \\
      9
    \end{ytableau}.
  \]
  Then,
  \[
    \operatorname{Sl}(T;2,1) = \begin{ytableau}
      1 & 1 & 2 & 2 & 3 \\
      *(lightgray) 3 & *(lightgray) 3 & *(lightgray) 8 & 8 \\
      6 & 6 & *(lightgray) 9 \\
      8 & 8 & *(lightgray) \bigcirc \\
      9
    \end{ytableau},
  \]
  where the shaded boxes represent the route of the hole.
\end{ex}

For a later use, we describe another algorithm which computes the sliding.

\begin{prop}\label{prop: sl route}
  Let $T \in \mathrm{SST}_n(\lambda,H)$ and $(r,c) \in H$.
  Let $S$ denote the punctured tableau obtained by the following algorithm.
  \begin{enumerate}
    \item Initialize $k := r-1$, $T_k := T$, $H_k := H$, $\gamma_k := c$.
    \item\label{step: 2 prop: sl route} Set
    \begin{align*}
      &\gamma_{k+1} := \min\{ j \geq \gamma_k \mid T_k(k+2,j) \leq T_k(k+1,j+1) \}, \\
      &H'_{k+1} := (H_k \setminus \{(k+1,\gamma_k)\}) \sqcup \{(k+1,\gamma_{k+1})\},
    \end{align*}
    and $T'_{k+1}$ to be the punctured tableau of shape $(\lambda,H'_{k+1})$ defined by
    \[
      T'_{k+1}(i,j) := \begin{cases}
        T_k(k+1,j+1) & \text{ if } i = k+1 \text{ and } j \in [\gamma_k,\gamma_{k+1}-1], \\
        T_k(i,j) & \text{ otherwise},
      \end{cases}
    \]
    for each $(i,j) \in D(\lambda,H'_{k+1})$.
    \item \begin{enumerate}
      \item If $(k+2,\gamma_{k+1}) \notin D(\lambda,H'_{k+1})$, then return $T'_{k+1}$.
      \item Otherwise, set
      \[
        H_{k+1} := (H'_{k+1} \setminus \{(k+1,\gamma_{k+1})\}) \sqcup \{(k+2,\gamma_{k+1})\},
      \]
      $T_{k+1}$ to be the punctured tableau of shape $(\lambda,H_{k+1})$ defined by
      \[
        T_{k+1}(i,j) := \begin{cases}
          T'_{k+1}(k+2,\gamma_{k+1}) & \text{ if } (i,j) = (k+1,\gamma_{k+1}), \\
          T'_{k+1}(i,j) & \text{ otherwise},
        \end{cases}
      \]
      for each $(i,j) \in D(\lambda,H_{k+1})$.
      Increment $k$ and then go back to step $\eqref{step: 2 prop: sl route}$.
    \end{enumerate}
  \end{enumerate}
  Then, we have
  \[
    S = \operatorname{Sl}(T;r,c).
  \]
  In particular, there exists a sequence $(\gamma_{r-1},\gamma_r,\dots,\gamma_k)$ of positive integers satisfying the following:
  \begin{enumerate}
    \item\label{item 1, prop: sl route} $c = \gamma_{r-1} \leq \cdots \leq \gamma_k$,
    \item\label{item 2, prop: sl route} $\gamma_i := \min\{ j \geq \gamma_{i-1} \mid T(i+1,j) \leq T(i,j+1) \}$ for all $i \in [r,k]$,
    \item $\operatorname{sh}(\operatorname{Sl}(T;r,c)) = (\lambda,H')$, where $H' := (H \setminus \{ (r,c) \}) \sqcup \{(k,\gamma_k)\})$,
    \item $(k,\gamma_k)$ is not a slidable hole of $(\lambda,H')$.
    \item for each $(i,j) \in D(\lambda, H')$, we have
    \[
      (\operatorname{Sl}(T;r,c))(i,j) = \begin{cases}
        T(i,j+1) & \text{ if } i \in [r,k] \text{ and } j \in [\gamma_{i-1},\gamma_i-1], \\
        T(i+1,j) & \text{ if } i \in [r,k-1] \text{ and } j = \gamma_i, \\
        T(i,j) & \text{ otherwise}.
      \end{cases}
    \]
  \end{enumerate}
\end{prop}
\begin{proof}
  The assertions are straightforwardly deduced from Definition \ref{def: sl}.
\end{proof}

\begin{defi}\label{def: sl route}
  We call the sequence $(\gamma_{r-1},\gamma_r,\dots,\gamma_k)$ in Proposition \ref{prop: sl route} the \emph{sliding route} for $(T;r,c)$.
  Also, we call $k$ the \emph{terminal row} of the sliding for $(T;r,c)$.
\end{defi}

\begin{ex}
  Let $T$ be as in Exercise \ref{ex: sl}.
  Then, the sliding route for $(T;2,1)$ is $(1,3,3,3)$, and the terminal row is $4$.
\end{ex}

Clearly, the sliding changes the row-word of a given tableau in general.
However, it does not change the Knuth equivalence class of the row-word when the given hole is the latest slidable hole with respect to the lexicographic order $\leq_\mathrm{lex}$, as we will see below.

\begin{lem}\label{lem: rows near latest sl hole}
  Assume that $(\lambda,H)$ has a slidable hole, and that the latest slidable hole $(r,c)$ with respect to $\leq_\mathrm{lex}$ satisfies $(r+1,c) \in D(\lambda,H)$.
  Set
  \begin{itemize}
    \item $p := c-1$,
    \item $q := \min\{ j \geq 0 \mid (r+1,c+j+1) \notin D(\lambda,H) \}$,
    \item $p' := \sharp\{ j \in [c-1] \mid (r,j) \in D(\lambda,H) \}$,
    \item $q' := \min\{ j \geq 0 \mid (r,c+j+1) \notin D(\lambda,H) \}$.
  \end{itemize}
  Then, the following hold:
  \begin{enumerate}
    \item $p \geq p'$,
    \item $q \leq q'$,
    \item $\{ j \geq 1 \mid (r+1,j) \in D(\lambda,H) \} = [p+q+1]$,
    \item $\{ j > c \mid (r,j) \in D(\lambda,H) \} = [c+1,c+q']$.
  \end{enumerate}
  Namely, the $r$-th and $(r+1)$-st rows of $D(\lambda,H)$ is of the following form:
  
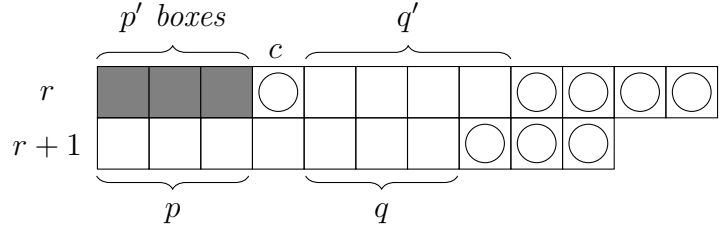
\[
    \begin{tikzpicture}
      \draw[fill=gray, scale=1.7] (-1.2,0) rectangle (-0.8,-0.4);
      \draw[fill=gray, scale=1.7] (-0.8,0) rectangle (-0.4,-0.4) node (v1) {};
      \draw[fill=gray, scale=1.7] (-0.4,0) rectangle (0,-0.4) node (v2) {};
      \draw[scale=1.7] (0,0) rectangle (0.4,-0.4) node (v3) {};
      \draw[scale=1.7] (0.4,0) rectangle (0.8,-0.4) node (v4) {};
      \draw[scale=1.7] (0.8,0) rectangle (1.2,-0.4) node (v5) {};
      \draw[scale=1.7] (1.2,0) rectangle (1.6,-0.4) node (v6) {};
      \draw[scale=1.7] (1.6,0) rectangle (2,-0.4) node (v7) {};
      \draw[scale=1.7] (2,0) rectangle (2.4,-0.4) node (v8) {};
      \draw[scale=1.7] (2.4,0) rectangle (2.8,-0.4);
      \draw[scale=1.7] (2.8,0) rectangle (3.2,-0.4);
      \draw[scale=1.7] (3.2,0) rectangle (3.6,-0.4);
      \draw[scale=1.7] (-1.2,-0.4) rectangle (-0.8,-0.8);
      \draw[scale=1.7] (-0.8,-0.4) rectangle (-0.4,-0.8);
      \draw[scale=1.7] (v1) rectangle (0,-0.8);
      \draw[scale=1.7] (v2) rectangle (0.4,-0.8);
      \draw[scale=1.7] (v3) rectangle (0.8,-0.8);
      \draw[scale=1.7] (v4) rectangle (1.2,-0.8);
      \draw[scale=1.7] (v5) rectangle (1.6,-0.8);
      \draw[scale=1.7] (v6) rectangle (2,-0.8);
      \draw[scale=1.7] (v7) rectangle (2.4,-0.8);
      \draw[scale=1.7] (v8) rectangle (2.8,-0.8);

      \draw[scale=1.7] (0.2,-0.2) circle (0.15);
      \draw[scale=1.7] (2.2,-0.2) circle (0.15);
      \draw[scale=1.7] (2.6,-0.2) circle (0.15);
      \draw[scale=1.7] (3,-0.2) circle (0.15);
      \draw[scale=1.7] (3.4,-0.2) circle (0.15);
      \draw[scale=1.7] (1.8,-0.6) circle (0.15);
      \draw[scale=1.7] (2.2,-0.6) circle (0.15);
      \draw[scale=1.7] (2.6,-0.6) circle (0.15);
      
      \draw [decorate, decoration={brace, amplitude=5pt}](-2.05,0.1)  --  node[above,yshift=0.5em]{$p'$ boxes} (-0.05,0.1);
      \draw [decorate, decoration={brace, amplitude=5pt}](0.7,0.1)  --  node[above,yshift=0.5em]{$q'$} (3.4,0.1);
      \draw [decorate, decoration={brace, amplitude=5pt, mirror}](-2.05,-1.45)  --  node[below,yshift=-0.5em]{$p$} (-0.05,-1.45);
      \draw [decorate, decoration={brace, amplitude=5pt, mirror}](0.7,-1.45) --  node[below,yshift=-0.5em]{$q$} (2.7,-1.45);
      \node at (0.3,0.2) {$c$};
      \node at (-2.7,-0.35) {$r$};
      \node at (-2.7,-1.05) {$r+1$};
    \end{tikzpicture}
  \]
\end{lem}
\begin{proof}
  By the definitions of $p$ and $p'$, we have
  \[
    p \geq p'.
  \]
  The other assertions follow from Lemma \ref{lem: rows of punc par} \eqref{item: sl lem: rows of punc par} and a similar argument to the proof of Proposition \ref{prop: punc par wo (rev)slidable holes} \eqref{item: 1 lem punc par wo (rev)slidable holes}.  
\end{proof}

\begin{prop}\label{prop: sl preserve Knu equiv}
  Let $T \in \mathrm{SST}_n(\lambda,H)$ and $(r,c) \in H$.
  If $(r,c)$ is the latest slidable hole with respect to $\leq_\mathrm{lex}$, then we have
  \[
    w_\mathrm{row}(T) \equiv w_\mathrm{row}(\operatorname{Sl}(T;r,c)).
  \]
\end{prop}
\begin{proof}
  Let us keep the notation in Definition \ref{def: sl}.
  We will prove by induction on $t \geq 0$ the following:
  if $(r_t,c_t)$ is the latest slidable hole of $(\lambda,H_t)$, then the following hold.
  \begin{enumerate}
    \item\label{item: 1 sl preserve Knu equiv} $w_\mathrm{row}(T_t) \equiv w_\mathrm{row}(T_{t+1})$.
    \item\label{item: 2 sl preserve Knu equiv} If $(r_{t+1},c_{t+1})$ is a slidable hole of $(\lambda,H_{t+1})$, then it is the latest slidable hole.
  \end{enumerate}

  Since $(r_{t+1},c_{t+1})$ equals either $(r_t,c_t+1)$ or $(r_t+1,c_t)$, we have
  \[
    (r_t,c_t) \leq_\mathrm{lex} (r_{t+1},c_{t+1}).
  \]
  Also, since
  \[
    H_{t+1} = (H_t \setminus \{(r_t,c_t)\}) \sqcup \{(r_{t+1},c_{t+1})\},
  \]
  each slidable hole $(k,l)$ of $(\lambda,H_{t+1})$ must satisfy one of the following conditions:
  \begin{itemize}
    \item $(k,l) = (r_{t+1},c_{t+1})$,
    \item $(k,l)$ is a slidable hole of $(\lambda,H_t)$,
    \item $(k,l) = \{ (r_t,c_t-1),\ (r_t-1,c_t) \}$.
  \end{itemize}
  Hence, claim \eqref{item: 2 sl preserve Knu equiv} follows.

  Let us prove the claim \eqref{item: 1 sl preserve Knu equiv}.
  If $(r_{t+1},c_{t+1}) = (r_t,c_t+1)$, then $T_t$ and $T_{t+1}$ have the same row-word.
  In this case, we have nothing to prove.
  Hence, assume that $(r_{t+1},c_{t+1}) = (r_t+1,c_t)$.
  By Lemma \ref{lem: rows near latest sl hole}, we see that
  \[
    w_\mathrm{row}(T_t) = w_1 v_{p-p'+1} \cdots v_p x z_1 \cdots z_q u_1 \cdots u_{p'} y_1 \cdots y_q w_2,
  \]
  where
  \begin{itemize}
    \item $w_1,w_2$ are some words,
    \item $p,q,p',q'$ are some nonnegative integers such that $p = c_t-1$, $p \geq p'$ and $q \leq q'$,
    \item $v_j = T_t(r_t+1,j)$ for all $j \in [p-p'+1,p]$,
    \item $x = T_t(r_t+1,c_t)$,
    \item $z_j = T_t(r_t+1,c_t+j)$ for all $j \in [q]$,
    \item $u_j = T_t(r_t,l_j)$ for all $j \in [p']$, where $1 \leq l_1 < \cdots < l_{p'} < c_t$ are such that
    \[
      \{ j < c_t \mid (r_t,j) \in D(\lambda,H) \} = \{ l_1,\dots,l_{p'} \},
    \]
    \item $y_j = T_t(r_t,c_t+j)$ for all $j \in [q]$.
  \end{itemize}
  Then, we can apply Lemma \ref{lem: sl lem} to obtain
  \[
    w_\mathrm{row}(T_t) \equiv w_1 v_{p-p'+1} \cdots v_p z_1 \cdots z_q u_1 \cdots u_{p'} x y_1 \cdots y_q w_2.
  \]
  The right-hand side is nothing but $w_\mathrm{row}(T_{t+1})$; see Remark \ref{rem: sl lem}.
  Thus, we complete the proof.
\end{proof}

\begin{defi}\label{defi: rect}
  The \emph{rectification} of $T \in \mathrm{SST}_n(\lambda,H)$ is the semistandard tableau $\operatorname{Rect}(T)$ obtained by the following algorithm.
  \begin{enumerate}
    \item Initialize $t := 0$, $T_t := T$, $H_t := H$.
    \item\label{step: 2 rect}
    \begin{enumerate}
      \item If $H_t$ has no slidable holes, then return $T_t$.
      \item Otherwise, let $(r,c) \in H_t$ denote the latest slidable hole with respect to $\leq_\mathrm{lex}$, and set $T_{t+1} := \operatorname{Sl}(T_t;r,c)$.
    \end{enumerate}
    \item Increment $t$ and then go back to step \eqref{step: 2 rect}.
  \end{enumerate}
\end{defi}

\begin{ex}\label{ex: rect}
  Let
  \[
    T := \begin{ytableau}
      1 & 1 & 2 & 2 & 3 \\
      \bigcirc & 3 & 3 & 8 \\
      \bigcirc & 6 & 8 \\
      6 & 8 & 9 \\
      8 \\
      9
    \end{ytableau}.
  \]
  Then,
  \[
    T \xrightarrow{\operatorname{Sl}(\cdot;3,1)} \begin{ytableau}
      1 & 1 & 2 & 2 & 3 \\
      \bigcirc & 3 & 3 & 8 \\
      *(lightgray) 6 & 6 & 8 \\
      *(lightgray) 8 & 8 & 9 \\
      *(lightgray) 9 \\
      *(lightgray) \bigcirc
    \end{ytableau} \xrightarrow{\operatorname{Sl}(\cdot;2,1)} \begin{ytableau}
      1 & 1 & 2 & 2 & 3 \\
      *(lightgray) 3 & *(lightgray) 3 & *(lightgray) 8 & 8 \\
      6 & 6 & *(lightgray) 9 \\
      8 & 8 & *(lightgray) \bigcirc \\
      9 \\
      \bigcirc
    \end{ytableau}.
  \]
  Hence,
  \[
    \operatorname{Rect}(T) = \begin{ytableau}
      1 & 1 & 2 & 2 & 3 \\
      3 & 3 & 8 & 8 \\
      6 & 6 & 9 \\
      8 & 8 \\
      9
    \end{ytableau}
  \]
\end{ex}

\begin{prop}\label{prop: row word and rect}
  Let $T \in \mathrm{SST}_n(\lambda,H)$.
  Then, we have
  \[
    w_\mathrm{row}(T) \equiv w_\mathrm{row}(\operatorname{Rect}(T)).
  \]
\end{prop}
\begin{proof}
  The assertion immediately follows from Proposition \ref{prop: sl preserve Knu equiv}.
\end{proof}

Proposition \ref{prop: row word and rect}, together with Proposition \ref{prop: row word equiv}, implies that $\operatorname{Rect}(T)$ is the unique semistandard tableau whose row-word is Knuth equivalent to $w_\mathrm{row}(T)$.
This is a slight generalization of \cite[\S 2.1 Corollary 1]{Ful97}.

\subsection{Reverse sliding and the reverse rectification}
Almost all the results in this subsection have a counterpart in the previous subsection, and their proofs are similar.
Hence, we omit such proofs here.

\begin{defi}\label{def: rev sl}
  The \emph{reverse sliding} is the algorithm which takes a pair $(T,(r,c)) \in \mathrm{SST}_n(\lambda,H) \times H$ as an input, and outputs a new semistandard punctured tableau $\operatorname{Sl}(T;r,c)$ as follows.
  \begin{enumerate}
    \item Initialize $t := 0$, $T_t := T$, $r_t := r$, $c_t := c$, $H_t := H$.
    \item\label{step: 2 rev sl} \begin{enumerate}
      \item If $(r_t,c_t)$ is not reversely slidable, then return $T_t$.
      \item If $T(r_t,c_t-1) > T(r_t-1,c_t)$, then set $r_{t+1} := r_t$, $c_{t+1} := c_t-1$.
      \item If $T(r_t,c_t-1) \leq T(r_t-1,c_t)$, then set $r_{t+1} := r_t-1$, $c_{t+1} := c_t$.
    \end{enumerate}
    \item Set $H_{t+1} := (H_t \setminus \{(r_t,c_t)\}) \sqcup \{(r_{t+1},c_{t+1})\}$, and $T_{t+1}$ to be the semistandard punctured tableau of shape $(\lambda,H_{t+1})$ defined by
    \[
      T_{t+1}(i,j) := \begin{cases}
        T_t(r_{t+1},c_{t+1}) & \text{ if } (i,j) = (r_t,c_t), \\
        T_t(i,j) & \text{ otherwise},
      \end{cases}
    \]
    for each $(i,j) \in D(\lambda,H_{t+1})$.
    \item Increment $t$ and then go back to step \eqref{step: 2 rev sl}.
  \end{enumerate}
\end{defi}

\begin{ex}
  Let
  \[
    T := \begin{ytableau}
      1 & 1 & 2 & 2 & 3 \\
      3 & 3 & 8 & 8 \\
      6 & 6 & 9 \\
      8 & 8 & \bigcirc \\
      9
    \end{ytableau}.
  \]
  Then,
  \[
    \operatorname{Sl}(T;2,1) = \begin{ytableau}
      *(lightgray) \bigcirc & 1 & 2 & 2 & 3 \\
      *(lightgray) 1 & *(lightgray) 3 & *(lightgray) 3 & 8 \\
      6 & 6 & *(lightgray) 8 \\
      8 & 8 & *(lightgray) 9 \\
      9
    \end{ytableau},
  \]
  where the shaded boxes represent the route of the hole.
\end{ex}

\begin{lem}\label{lem: rows near first rev sl hole}
  Assume that $H$ has a reversely slidable hole, and that the first reversely slidable hole $(r,c)$ with respect to $\leq_\mathrm{lex}$ satisfies $(r-1,c) \in D(\lambda,H)$.
  Set
  \begin{itemize}
    \item $p := \min\{ j \geq 0 \mid (r-1,c-j-1) \notin D(\lambda,H) \}$,
    \item $q := \lambda_{r-1} - c$,
    \item $p' := \min\{ j \geq 0 \mid (r,c-j-1) \notin D(\lambda,H) \}$.
    \item $q' := \sharp\{ j > c \mid (r,j) \in D(\lambda,H) \}$,
  \end{itemize}
  Then, the following hold:
  \begin{enumerate}
    \item $p \leq p'$,
    \item $q \geq q'$,
    \item $\{ j \geq 1 \mid (r-1,j) \in D(\lambda,H) \} = [c-p,c+q]$,
    \item $\{ j < c \mid (r,j) \in D(\lambda,H) \} = [c-p',c-1]$.
  \end{enumerate}
  Namely, the $(r-1)$-st and $r$-th rows of $D(\lambda,H)$ is of the following form:
  \[
    \begin{tikzpicture}
      \draw[scale=1.7] (-1.6,0.6) rectangle (-1.2,0.2) node (v1) {};
      \draw[scale=1.7] (-1.2,0.6) rectangle (-0.8,0.2) node (v2) {};
      \draw[scale=1.7] (-0.8,0.6) rectangle (-0.4,0.2);
      \draw[scale=1.7] (-0.4,0.6) rectangle (0,0.2);
      \draw[scale=1.7] (0,0.6) rectangle (0.4,0.2);
      \draw[scale=1.7] (0.4,0.6) rectangle (0.8,0.2);
      \draw[scale=1.7] (0.8,0.6) rectangle (1.2,0.2);
      \draw[scale=1.7] (1.2,0.6) rectangle (1.6,0.2);
      \draw[scale=1.7] (1.6,0.6) rectangle (2,0.2);
      \draw[scale=1.7] (2,0.6) rectangle (2.4,0.2);
      \draw[scale=1.7] (2.4,0.6) rectangle (2.8,0.2);
      \draw[scale=1.7] (2.8,0.6) rectangle (3.2,0.2);
      \draw[scale=1.7] (3.2,0.6) rectangle (3.6,0.2);
      \draw[scale=1.7] (3.6,0.6) rectangle (4,0.2);
      \draw[scale=1.7] (-1.6,0.2) rectangle (-1.2,-0.2);
      \draw[scale=1.7] (v1) rectangle (-0.8,-0.2);
      \draw[scale=1.7] (v2) rectangle (-0.4,-0.2);
      \draw[scale=1.7] (-0.4,0.2) rectangle (0,-0.2);
      \draw[scale=1.7] (0,0.2) rectangle (0.4,-0.2);
      \draw[scale=1.7] (0.4,0.2) rectangle (0.8,-0.2);
      \draw[scale=1.7] (0.8,0.2) rectangle (1.2,-0.2);
      \draw[scale=1.7] (1.2,0.2) rectangle (1.6,-0.2);
      \draw[scale=1.7] (1.6,0.2) rectangle (2,-0.2);
      \draw[fill=gray,scale=1.7] (2,0.2) rectangle (2.4,-0.2);
      \draw[fill=gray,scale=1.7] (2.4,0.2) rectangle (2.8,-0.2);
      \draw[fill=gray,scale=1.7] (2.8,0.2) rectangle (3.2,-0.2);

      \draw[scale=1.7] (-1.4,0.4) circle (0.15);
      \draw[scale=1.7] (-1,0.4) circle (0.15);
      \draw[scale=1.7] (-0.6,0.4) circle (0.15);
      \draw[scale=1.7] (-0.2,0.4) circle (0.15);
      \draw[scale=1.7] (0.2,0.4) circle (0.15);
      \draw[scale=1.7] (-1.4,0) circle (0.15);
      \draw[scale=1.7] (-1,0) circle (0.15);
      \draw[scale=1.7] (-0.6,0) circle (0.15);
      \draw[scale=1.7] (1.8,0) circle (0.15);
      
      \draw [decorate, decoration={brace, amplitude=5pt}](0.7,1.1)  --  node[above,yshift=0.5em]{$p$} (2.7,1.1);
      \draw [decorate, decoration={brace, amplitude=5pt}](3.45,1.1)  --  node[above,yshift=0.5em]{$q$} (6.75,1.1);
      \draw [decorate, decoration={brace, amplitude=5pt, mirror}](-0.65,-0.45)  --  node[below,yshift=-0.5em]{$p'$} (2.7,-0.45);
      \draw [decorate, decoration={brace, amplitude=5pt, mirror}](3.45,-0.45) --  node[below,yshift=-0.5em]{$q'$ boxes} (5.45,-0.45);

      \node at (3.05,1.2) {$c$};
      \node at (-3.5,0.65) {$r-1$};
      \node at (-3.5,-0.05) {$r$};
\end{tikzpicture}
  \]
\end{lem}

\begin{prop}\label{prop: rev sl preserve Knu equiv}
  Let $T \in \mathrm{SST}_n(\lambda,H)$ and $(r,c) \in H$.
  If $(r,c)$ is the first reversely slidable hole with respect to $\leq_\mathrm{lex}$, then we have
  \[
    w_\mathrm{row}(T) \equiv w_\mathrm{row}(\operatorname{revSl}(T;r,c)).
  \]
\end{prop}

\begin{defi}\label{defi: rev rect}
  The \emph{reverse rectification} of $T \in \mathrm{SST}_n(\lambda,H)$ is the semistandard skew tableau $\operatorname{revRect}(T)$ obtained by the following algorithm.
  \begin{enumerate}
    \item Initialize $t := 0$, $T_t := T$, $H_t := H$.
    \item\label{step: 2 rev rect}
    \begin{enumerate}
      \item If $H_t$ has no reversely slidable holes, then return $T_t$.
      \item Otherwise, let $(r,c) \in H_t$ denote the first reversely slidable hole with respect to $\leq_\mathrm{lex}$, and set $T_{t+1} := \operatorname{revSl}(T_t;r,c)$.
    \end{enumerate}
    \item Increment $t$ and then go back to step \eqref{step: 2 rev rect}.
  \end{enumerate}
\end{defi}

\begin{ex}
  Let
  \[
    T := \begin{ytableau}
      1 & 1 & 2 & 2 & 3 \\
      3 & 3 & 8 & 8 \\
      6 & 6 & 9 \\
      8 & 8 & \bigcirc \\
      9 \\
      \bigcirc
    \end{ytableau}.
  \]
  Then,
  \[
    T \xrightarrow{\operatorname{revSl}(\cdot;4,3)} \begin{ytableau}
      *(lightgray) \bigcirc & 1 & 2 & 2 & 3 \\
      *(lightgray) 1 & *(lightgray) 3 & *(lightgray) 3 & 8 \\
      6 & 6 & *(lightgray) 8 \\
      8 & 8 & *(lightgray) 9 \\
      9 \\
      \bigcirc
    \end{ytableau} \xrightarrow{\operatorname{revSl}(\cdot;6,1)} \begin{ytableau}
      \bigcirc & 1 & 2 & 2 & 3 \\
      *(lightgray) \bigcirc & 3 & 3 & 8 \\
      *(lightgray) 1 & 6 & 8 \\
      *(lightgray) 6 & 8 & 9 \\
      *(lightgray) 8 \\
      *(lightgray) 9
    \end{ytableau}.
  \]
  Hence,
  \[
    \operatorname{revRect}(T) = \begin{ytableau}
      \none & 1 & 2 & 2 & 3 \\
      \none & 3 & 3 & 8 \\
      1 & 6 & 8 \\
      6 & 8 & 9 \\
      8 \\
      9
    \end{ytableau}.
  \]
\end{ex}

\begin{prop}\label{prop: row word and rev rect}
  Let $T \in \mathrm{SST}_n(\lambda,H)$.
  Then, we have
  \[
    w_\mathrm{row}(T) \equiv w_\mathrm{row}(\operatorname{revRect}(T)).
  \]
\end{prop}

\begin{cor}\label{cor: rect = rect rev rect}
  Let $T \in \mathrm{SST}_n(\lambda,H)$.
  Then, we have
  \[
    \operatorname{Rect}(T) = \operatorname{Rect}(\operatorname{revRect}(T));
  \]
  here we regard $\operatorname{revRect}(T)$ as a punctured tableau by identifying its shape, which is a skew partition, as a punctured partition\rm{;} see \rm{Remark \ref{rem: skew par is punc par}}.
\end{cor}
\begin{proof}
  By Propositions \ref{prop: row word and rect} and \ref{prop: row word and rev rect}, we have
  \[
    w_\mathrm{row}(\operatorname{Rect}(T)) \equiv w_\mathrm{row}(T) \equiv w_\mathrm{row}(\operatorname{revRect}(T)) \equiv w_\mathrm{row}(\operatorname{Rect}(\operatorname{revRect}(T))).
  \]
  Then, Proposition \ref{prop: row word equiv} implies the assertion.
\end{proof}

\begin{cor}\label{cor: rect and col-word}
  Assume that the holes of $(\lambda,H)$ lie in the first column:
  \[
    H \subset \{ (r,1) \mid r \geq 1 \}.
  \]
  Then, for each $T \in \mathrm{SST}_n(\lambda,H)$, we have
  \[
    w_\mathrm{row}(\operatorname{Rect}(T)) \equiv w_\mathrm{col}(T).
  \]
\end{cor}
\begin{proof}
  By Corollary \ref{cor: rect = rect rev rect} and Proposition \ref{prop: row word and rect}, we have
  \[
    w_\mathrm{row}(\operatorname{Rect}(T)) = w_\mathrm{row}(\operatorname{Rect}(\operatorname{revRect}(T))) \equiv w_\mathrm{row}(\operatorname{revRect}(T)).
  \]
  Then, Proposition \ref{prop: row col} implies that
  \[
    w_\mathrm{row}(\operatorname{revRect}(T)) \equiv w_\mathrm{col}(\operatorname{revRect}(T)).
  \]
  By our assumption, it is clear that $w_\mathrm{col}(T) = w_\mathrm{col}(\operatorname{revRect}(T))$.
  Therefore, the assertion follows.
\end{proof}

\section{Berele row-insertion}\label{sect: ber row-ins}
In this section, we study the Berele row-insertion in detail.
The original definition of the Berele row-insertion involves the Schensted row-insertion and sliding.
We rewrite it in terms of the Knuth equivalence.

Throughout this section, we fix a partition $\nu \in \mathrm{Par}_{\leq n}$.

\subsection{Symplectic tableaux}
\begin{defi}[{\cite[\S 4]{Kin76}}]\label{def: spt}
  A semistandard tableau $T \in \mathrm{SST}_{2n}(\nu)$ is said to be \emph{symplectic} if
  \[
    T(i,1) \geq 2i-1 \quad \text{ for all } i \geq 1.
  \]
  Let $Sp\mathrm{T}_{2n}(\nu)$ denote the set of symplectic tableaux in $\mathrm{SST}_{2n}(\nu)$.
\end{defi}

\begin{lem}\label{lem: symp-ness of row ins}
  Let $T \in Sp\mathrm{T}_{2n}(\nu)$ and $x \in [2n]$.
  Set $S := T \leftarrow x$.
  Let $(x_1,\dots,x_r)$ denote the row-inserting letters for $(T,x)$.
  Then, the following hold:
\begin{enumerate}
  \item\label{item 1 lem: symp-ness of row ins} If $x_i \geq 2i-1$ for all $i \in [r]$, then $S$ is symplectic.
  \item\label{item 2 lem: symp-ness of row ins} If $x_i < 2i-1$ for some $i \in [r]$, then $S$ is not symplectic.
  Moreover, if $s$ denotes the minimum among those $i$, then the following hold:
  \begin{enumerate}
    \item $s \geq 2$,
    \item $x_{s-1} = 2s-3$,
    \item $x_s = 2s-2$,
    \item $x_i = T(i-1,1)$ for all $i \in [s+1,r]$.
  \end{enumerate}
\end{enumerate}
\end{lem}
\begin{proof}
  Let us prove the first assertion.
  By Proposition \ref{prop: explicit desc row ins} \eqref{item: 7 prop: explicit desc row ins}, we have
  \[
    S(i,1) \in \{ x_i, T(i,1) \} \quad \text{ for all } i \in [r],
  \]
  and
  \[
    S(i,1) = T(i,1) \quad \text{ for all } i > r.
  \]
  Since $T$ is symplectic, we have $T(i,1) \geq 2i-1$ for all $i \geq 1$.
  Hence, we see that
  \[
    S(i,1) \geq 2i-1 \quad \text{ for all } i \geq 1.
  \]
  This implies that $S$ is symplectic.

  Let us prove the second assertion.
  Let $(c_1,\dots,c_r)$ denote the row-insertion route for $(T,x)$.
  Since $x_s \geq 1 = 2 \cdot 1 - 1$, we obtain $s \geq 2$.
  By the definition of $s$ and Proposition \ref{prop: explicit desc row ins} \eqref{item: 5 prop: explicit desc row ins}, we obtain
  \[
    2s-1 > x_s > x_{s-1} \geq 2(s-1)-1 = 2s-3.
  \]
  This implies that
  \[
    x_{s-1} = 2s-3, \quad x_s = 2s-2.
  \]
  By Proposition \ref{prop: explicit desc row ins} \eqref{item: 2 prop: explicit desc row ins}, we have  
  \[
    c_s = \min\{ j \geq 1 \mid T(s,j) > x_s \}.
  \]
  Since $T$ is symplectic, it holds that $T(s,1) \geq 2s-1 > x_s$.
  Hence, we obtain $c_s = 1$.
  Then, Proposition \ref{prop: explicit desc row ins} \eqref{item: 3 prop: explicit desc row ins} implies that $c_i = 1$ for all $i \in [s,r]$.
  By Proposition \ref{prop: explicit desc row ins} \eqref{item: 4 prop: explicit desc row ins}, we obtain
  \[
    x_i = T(i-1,c_{i-1}) = T(i-1,1) \quad \text{ for all } i \in [s+1,r].
  \]
  Thus, we complete the proof.
\end{proof}

\subsection{Berele row-insertion}
\begin{defi}[{\cite[\S 2]{Ber86}}]\label{def: ber row ins}
  The \emph{Berele row-insertion} is the algorithm which takes a pair $(T,x) \in Sp\mathrm{T}_{2n}(\nu) \times [2n]$ as input, and outputs a new symplectic tableau $T \xleftarrow{\mathrm{B}} x$ as follows.
  \begin{enumerate}
    \item Initialize $r := 1$, $T_r := T$, $x_r := x$.
    \item\label{step 2, ber row ins} Set $c_r := \min\{ j \geq 1 \mid T(r,j) > x_r \}$.
    \item \begin{enumerate}
      \item If $c_r = \nu_r+1$, then return $\operatorname{add}(T_r,r,x_r)$.
      \item If $c_r \leq \nu_r$, $x_r = 2r-1$, and $T(r,c_r) = 2r$, then set $S$ to be the semistandard punctured tableau of shape $(\nu, \{(r,1)\})$ defined by
      \[
        S(i,j) := \begin{cases}
          2r-1 & \text{ if } (i,j) = (r,c_r), \\
          T_r(i,j) & \text{ otherwise},
        \end{cases}
      \]
      for each $(i,j) \in D(\nu,\{(r,1)\})$.
      Then, return $\operatorname{Rect}(S)$.
      \item Otherwise, set $T_{r+1}$ to be the symplectic tableau of shape $\lambda$ obtained from $T_r$ by replacing the $(r,c_r)$-entry with $x_r$:
      \[
        T_{r+1}(i,j) := \begin{cases}
          x_r & \text{ if } (i,j) = (r,c_r), \\
          T_r(i,j) & \text{ otherwise},
        \end{cases}
      \]
      for each $(i,j) \in D(\nu)$.
      Also, set $x_{r+1} := T_r(r,c_r)$, increment $r$, and then go back to step \eqref{step 2, ber row ins}.
    \end{enumerate}
  \end{enumerate}
\end{defi}

\begin{thm}[\textit{cf}.\,{\cite[Theorem 2]{Ber86}}]\label{thm: ber row ins}
  Let $\nu \in \mathrm{Par}_{\leq n}$.
  The Berele row-insertion gives rise to a bijection
  \[
    Sp\mathrm{T}_{2n}(\nu) \times [2n] \to \bigsqcup_{\substack{\xi \in \mathrm{Par}_{\leq n} \\ \xi \supset \nu,\ |\xi/\nu| = 1}} Sp\mathrm{T}_{2n}(\xi) \sqcup \bigsqcup_{\substack{\xi \in \mathrm{Par}_{\leq n} \\ \xi \subset \nu, \ |\nu/\xi| = 1}} Sp\mathrm{T}_{2n}(\xi).
  \]
\end{thm}

\begin{lem}\label{lem: ber row ins route}
  Let $T \in Sp\mathrm{T}_{2n}(\nu)$ and $x \in [2n]$.
  Let $(x_1,\dots,x_r)$ denote the row-inserting letters for $(T,x)$.
  Set $S := T \xleftarrow{\mathrm{B}} x$ and $\xi := \operatorname{sh}(S)$.
  \begin{enumerate}
    \item\label{item 1 lem: ber row ins route} Suppose that $x_i \geq 2i-1$ for all $i \in [r]$.
    Then, we have
    \begin{enumerate}
      \item $S = T \leftarrow x$,
      \item $\xi = \operatorname{add}(\nu,r)$.
    \end{enumerate}
    \item\label{item 2 lem: ber row ins route} Suppose that $x_i < 2i-1$ for some $i \in [r]$.
    Let $s$ denote the minimum among those $i$, and set $T'$ to be the punctured tableau of shape $(\nu, \{(s-1,1)\})$ defined by
    \[
      T'(i,j) := \begin{cases}
        x_i & \text{ if } i \in [s-1] \text{ and } j = c_i, \\
        T(i,j) & \text{ otherwise},
      \end{cases}
    \]
    for each $(i,j) \in D(\nu, \{(s-1,1)\})$.
    Then, we have
    \begin{enumerate}
      \item $S = \operatorname{Rect}(T')$,
      \item $\xi = \operatorname{rm}(\nu,k)$, where $k$ denotes the terminal row of the sliding for $(T';s-1,1)$.
    \end{enumerate}
  \end{enumerate}
\end{lem}
\begin{proof}
  The assertion follows from Definition \ref{def: ber row ins}, Proposition \ref{prop: explicit desc row ins} \eqref{item: 6 prop: explicit desc row ins}, and Lemma \ref{lem: symp-ness of row ins}.
\end{proof}

Let $\mathcal{R} := \mathbb{Z}_{> 0} \sqcup \{ \bar{r} \mid r \in \mathbb{Z}_{> 0} \}$ denote the totally ordered set such that
\[
  1 < 2 < \cdots < \bar{2} < \bar{1}.
\]
An element of $\mathcal{R}$ is said to be \emph{unbarred} if it belongs to $\mathbb{Z}_{> 0}$, and \emph{barred} otherwise.

\begin{defi}\label{def: ber row ins route}
  Let us keep the notation in Lemma \ref{lem: ber row ins route}.
  The \emph{terminal row} of the Berele row-insertion for $(T,x)$ is $r \in \mathcal{R}$ if $x_i \geq 2i-1$ for all $i \in [r]$; otherwise $\bar{k} \in \mathcal{R}$.
\end{defi}

\begin{defi}
  Let us keep the notation in Lemma \ref{lem: ber row ins route}.
  Suppose that $x_i < 2i-1$ for some $i \in [r]$, and let $(\gamma_{s-2},\gamma_{s-1},\dots,\gamma_k)$ denote the sliding route for $(T';s-1,1)$.
  We call $(c_1,\dots,c_{s-1};\gamma_{s-1},\dots,\gamma_{k})$ the Berele row-insertion route for $(T,x)$.
\end{defi}

\begin{ex}
  Let
  \[
    T := \begin{ytableau}
      1 & 1 & 2 & 3 & 3 \\
      3 & 3 & 4 & 8 \\
      6 & 6 & 8 \\
      8 & 8 & 9 \\
      9
    \end{ytableau}.
  \]
  Then,
  \[
    T \xleftarrow{\mathrm{B}} 2 = \begin{ytableau}
      1 & 1 & 2 & *(lightgray) 2 & 3 \\
      3 & 3 & *(lightgray) 8 & 8 \\
      6 & 6 & *(lightgray) 9 \\
      8 & 8 & *(lightgray) \bigcirc \\
      9
    \end{ytableau},
  \]
  where the shaded boxes represent the Berele row-insertion route $(4,3;3,3,3)$.
  The terminal row is $\bar{4}$.
\end{ex}

\begin{lem}[Berele row-insertion lemma]\label{lem: ber row ins}
  Let $T \in Sp\mathrm{T}_{2n}(\nu)$, $x,x' \in [2n]$, and set $T' := T \xleftarrow{\mathrm{B}} x$, $\nu' := \operatorname{sh}(T')$, $T'' := T' \xleftarrow{\mathrm{B}} x''$, $\nu'' := \operatorname{sh}(T'')$.
  Let $r,r'$ denote the terminal rows of the Berele row-insertion for $(T,x)$ and $(T',x')$, respectively.
  \begin{enumerate}
    \item\label{item: row, lem: ber row ins} If $x \leq x'$, then we have $r \geq r'$
    \item\label{item: col, lem: ber row ins} If $x > x'$, then we have $r < r'$
  \end{enumerate}
\end{lem}
\begin{proof}
  The first assertion has been proved in \cite[Lemma 10.4]{Sun86}.

  Let us prove the second assertion.
  Suppose that $\nu \supset \nu'$.
  Let $(c_1,\dots,c_{s-1};\gamma_{s-1},\dots,\gamma_k)$ denote the Berele row-insertion route for $(T,x)$ and $(x_1,\dots,x_s,\dots)$ the row-inserting letters for $(T,x)$.
  Then, we have
  \[
    T'(i,j) = \begin{cases}
      x_i & \text{ if } i \in [s-2] \text{ and } j = c_i, \\
      2s-3 & \text{ if } i = s-1 \text{ and } j \in [c_{s-1}-1], \\
      T(s-1,j+1) & \text{ if } i = s-1 \text{ and } j \in [c_{s-1},\gamma_{s-1}-1], \\
      T(i,j+1) & \text{ if } i \in [s,k] \text{ and } j \in [\gamma_{i-1},\gamma_i-1], \\
      T(i+1,j) & \text{ if } i \in [s-1,k-1] \text{ and } j = \gamma_i, \\
      T(i,j) & \text{ otherwise},
    \end{cases}
  \]
  for all $(i,j) \in D(\nu')$.

  Set $S := T' \leftarrow x'$, and let $(c''_1,\dots,c''_{r''})$ and $(x''_1,\dots,x''_{r''})$ denote the row-insertion route and the row-inserting letters for $(T',x')$, respectively.
  Then, we have
  \[
    x''_1 = x' < x = x_1
  \]
  and
  \[
    c''_1 = \min\{ j \geq 1 \mid T'(1,j) > x''_1 \}.
  \]
  Since $T'(1,c_1) = x_1 > x''_1$, it holds that $c''_1 \leq c_1$.
  Proceeding in this way, one can deduce that
  \begin{align}\label{ineq: x''}
    x''_{s-1} < x_{s-1} = 2s-3 = 2(s-1)-1.
  \end{align}
  By Lemma \ref{lem: ber row ins route} \eqref{item 2 lem: ber row ins route}, this implies that $\nu' \supset \nu''$.
  Let $(c'_1,\dots,c'_{s'-1};\gamma'_{s'-1},\dots,\gamma'_{k'})$ denote the Berele row-insertion route for $(T',x')$.
  By inequality \eqref{ineq: x''}, we see that
  \[
    s' < s.
  \]

  Let us show that $\gamma_i \leq \gamma'_i$ for all $i \in [s-1,k] \cap [s'-1,k']$.
  Assume contrary that $\gamma_i > \gamma'_i$ for some $i$.
  We may further assume that this $i$ is the minimum among those.
  Since $s' < s$, we must have $i \geq s'$.
  By the minimality of $i$, it holds that $\gamma_{i-1} \leq \gamma'_{i-1}$, where we set $\gamma_{s-2} := 1$.
  Hence, by Proposition \ref{prop: sl route} \eqref{item 1, prop: sl route}, we have
  \[
    \gamma_{i-1} \leq \gamma'_{i-1} \leq \gamma'_i < \gamma_i.
  \]
  Then, we obtain
  \begin{align*}
    &T''(i-1,\gamma'_{i-1}) = T'(i,\gamma'_{i-1}) = T(i,\gamma'_{i-1}+1), \\
    &T''(i-1,\gamma'_{i-1}+1) = T'(i-1,\gamma'_{i-1}+1) = T(i-1,\gamma'_{i-1}+1).
  \end{align*}
  However, these contradict that $T''$ and $T$ are semistandard:
  \[
    T''(i-1,\gamma'_{i-1}) \leq T''(i-1, \gamma'_{i-1}), \quad T(i,\gamma'_{i-1}+1) > T(i-1,\gamma'_{i-1}+1).
  \]
  
  Next, let us prove that $k > k'$.
  Assume contrary that $k \leq k'$.
  We have proved that $\gamma_k \leq \gamma'_k$.
  On the other hand, by Lemma \ref{lem: ber row ins route} \eqref{item 2 lem: ber row ins route}, we must have $\gamma_k = \nu_k = \nu'_k+1$.
  However, Proposition \ref{prop: sl route} \eqref{item 2, prop: sl route} implies that $\gamma'_k \leq \nu'_k$.
  Therefore, we obtain a contradiction.

  So far, we have proved the following.
  If $\nu \supset \nu'$, then we have $\nu' \supset \nu''$.
  Moreover, it holds that $r = \bar{k} < \bar{k'} = r'$.

  Now, assume that $\nu' \subset \nu''$.
  In this case, we must have $\nu \subset \nu'$, for otherwise what we just stated yields $\nu' \supset \nu''$.
  By Lemma \ref{lem: ber row ins route} \eqref{item 1 lem: ber row ins route}, we obtain
  \[
    T' = T \leftarrow x, \quad T'' = T' \leftarrow x'.
  \]
  Then, the assertion follows from \cite[\S 1.1 Row Bumping Lemma (2)]{Ful97}.

  It remains to consider the case when $\nu \subset \nu' \supset \nu''$.
  By Lemma \ref{lem: ber row ins route}, we see that $r$ is unbarred and $r'$ is barred.
  Hence, we obtain $r < r'$.
  Thus, we complete the proof.
\end{proof}

\section{Row-insertion of type $A\mathrm{II}$}\label{sect: row-ins aii}
In this section, we briefly recall results in \cite{Wat23f}, and introduce the row-insertion of type $A\mathrm{II}$.
Then, we show that it coincides with the Berele row-insertion.

Let $q$ be an indeterminate, and
\[
  \mathbf{A}_\infty := \mathbb{Q}[\![q^{-1}]\!] \cap \mathbb{Q}(q)
\]
denote the ring of rational functions that are regular at $q = \infty$.
Define an equivalence relation $\equiv_\infty$ on $\mathbb{Q}(q)$ by declaring $f \equiv_\infty g$ to mean that $f-g \in q^{-1} \mathbf{A}_\infty$.

\subsection{Quantum groups}
Let $\mathbf{U}$ denote the quantum group of $\mathfrak{gl}_{2n}$ (\cite[3.1.1 and Corollary 33.1.5]{Lus93}).
Namely, it is the unital associative $\mathbb{Q}(q)$-algebra with generators
\[
  \{ E_i,F_i,D_k^{\pm 1} \mid i \in [2n-1],\ k \in [2n] \}
\]
subject to the following relations:
for each $i,j \in [2n-1]$ and $k,l \in [2n]$, we have
\begin{align*}
  &D_k D_k^{-1} = D_k^{-1}D_k = 1,\\
  &D_k D_l = D_l D_k, \\
  &D_k E_i = q^{\delta_{k,i} - \delta_{k,i+1}} E_i D_k, \\
  &D_k F_i = q^{-\delta_{k,i} + \delta_{k,i+1}} F_i D_k, \\
  &E_i F_j - F_j E_i = \delta_{i,j} \frac{K_i - K_i^{-1}}{q - q^{-1}}, \\
  &E_i E_j = E_j E_i \quad \text{ if } |i-j| > 1, \\
  &F_i F_j = F_j F_i \quad \text{ if } |i-j| > 1, \\
  &E_i^2 E_j - (q+q^{-1}) E_i E_j E_i + E_j E_i^2 = 0 \quad \text{ if } |i-j| = 1, \\
  &F_i^2 F_j - (q+q^{-1}) F_i F_j F_i + F_j F_i^2 = 0 \quad \text{ if } |i-j| = 1,
\end{align*}
where
\[
  K_i := D_i D_{i+1}^{-1}.
\]

Let $\rho$ denote the anti-algebra involution on $\mathbf{U}$ (\cite[19.1.1]{Lus93}) defined by
\[
  \rho(E_i) := q^{-1} F_i K_i, \ \rho(F_i) := q K_i^{-1} E_i, \ \rho(D_k) := D_k, \quad \text{ for all } i \in [2n-1],\ k \in [2n].
\]

A symmetric bilinear form $(,)$ on a $\mathbf{U}$-module is said to be \emph{contragredient} if
\[
  (xu,v) = (u,\rho(x)v) \quad \text{ for all } x \in \mathbf{U},\ u,v \in M.
\]
A basis $B$ of a $\mathbf{U}$-module $M$ equipped with a contragredient bilinear form $(,)$ is said to be \emph{almost orthonormal} if
\[
  (b_1,b_2) \equiv_\infty \delta_{b_1,b_2} \quad \text{ for all } b_1,b_2 \in B.
\]

Given a $\mathbf{U}$-module $M$ equipped with a contragredient bilinear form $(,)$ and an almost orthonormal basis $B$, we define an equivalence relation $\equiv_\infty$ on $M$ by declaring $u \equiv_\infty v$ to mean
\[
  (u,b) \equiv_\infty (v,b) \quad \text{ for all } b \in B.
\]

The quantum group $\mathbf{U}$ is equipped with a Hopf algebra structure as in \cite[\S 3.3]{Lus93}.
Let $M,N$ be $\mathbf{U}$-modules equipped with contragredient bilinear forms $(,)_M$, $(,)_N$ and almost orthonormal bases $B_M$, $B_N$, respectively.
Then, the tensor product module $M \otimes N$ has a contragredient bilinear form $(,)$ defined by
\[
  (m_1 \otimes n_1, m_2 \otimes n_2) := (m_1,m_2)_M \cdot (n_1,n_2)_N \quad \text{ for all } m_1,m_2 \in M,\ n_1,n_2 \in N.
\]
The basis $\{ b \otimes b' \mid b \in B_M,\ b' \in B_N \}$ of $M \otimes N$ is almost orthonormal with respect to this bilinear form.

For each $\lambda \in \mathrm{Par}_{\leq 2n}$, there exists a unique finite-dimensional simple $\mathbf{U}$-module $V(\lambda)$ of highest weight $\lambda$.
It has an almost orthonormal basis, called the \emph{canonical basis} (\textit{cf}.\,\cite[Definition 14.4.12]{Lus93}), of the form
\[
  \{ b_T \mid T \in \mathrm{SST}_{2n}(\lambda) \}.
\]

In what follows, we often identify each element $T \in \mathrm{SST}_{2n}(1)$ with its unique entry $T(1,1) \in [2n]$.
In particular, the basis elements of $V(1)$ are denoted by $b_x$ with $x \in [2n]$.

\begin{thm}[\textit{cf}.\,{\cite{Kwo09b}}]\label{thm: sst and qg}
  Let $\lambda \in \mathrm{Par}_{\leq 2n}$.
  \begin{enumerate}
    \item\label{item: row ins, thm: sst and qg} There exists a $\mathbf{U}$-module isomorphism
    \[
      \mathrm{RI} : V(\lambda) \otimes V(1) \to \bigoplus_{\substack{\mu \in \mathrm{Par}_{\leq 2n} \\ \mu \supset \lambda,\ |\mu/\lambda| = 1}} V(\mu)
    \]
    such that
    \[
      \mathrm{RI}(b_T \otimes b_x) \equiv_\infty b_{T \leftarrow x} \quad \text{ for all } T \in \mathrm{SST}_{2n}(\lambda),\ x \in [2n].
    \]
    \item\label{item: row word, thm: sst and qg} Set $N := |\lambda|$.
    There exists an injective $\mathbf{U}$-module homomorphism
    \[
      \mathrm{RR} : V(\lambda) \to V(1)^{\otimes N}
    \]
    satisfying the following for all $T \in \mathrm{SST}_{2n}$\rm{:}
    if $w_\mathrm{row}(T) = (x_1,\dots,x_N)$, then we have
    \[
      \mathrm{RR}(T) \equiv_\infty b_{x_1} \otimes \cdots \otimes b_{x_N}.
    \]
  \end{enumerate}
\end{thm}

\subsection{Quantum symmetric pairs}
For each $i \in [n-1]$, set
\[
  B_{2i} := F_{2i} - q[E_{2i-1}, [E_{2i+1}, E_{2i}]_{q^{-1}}]_{q^{-1}} K_{2i}^{-1} \in \mathbf{U},
\]
where $[,]_{q^{-1}}$ denotes the $q^{-1}$-commutator:
\[
  [x,y]_{q^{-1}} := xy - q^{-1}yx.
\]
Let $\mathbf{U}^\imath$ denote the subalgebra of $\mathbf{U}$ generated by
\[
  \{ E_{2j-1}, F_{2j-1}, K_{2j-1}^{\pm 1} \mid i \in [n] \} \sqcup \{ B_{2k} \mid k \in [n-1] \}.
\]
The pair $(\mathbf{U},\mathbf{U}^\imath)$ forms a quantum symmetric pair of type $A\mathrm{II}_{2n-1}$.

The subalgebra $\mathbf{U}^\imath$ is invariant under the involution $\rho$ \cite[Proposition 4.6]{BaWa18}.
Hence, we can define the notions of contragredient bilinear forms and almost orthonormal bases, and equivalence relations $\equiv_\infty$ on modules, just as in quantum group.

Since $\mathbf{U}^\imath$ is a subalgebra of $\mathbf{U}$, each $\mathbf{U}$-module $M$ can be regarded as a $\mathbf{U}^\imath$-module by restriction.
If $(,)$ is a contragredient bilinear form of the $\mathbf{U}$-module $M$, then it is also a contragredient bilinear form of the $\mathbf{U}^\imath$-module $M$.

The subalgebra $\mathbf{U}^\imath$ is a right coideal of $\mathbf{U}$.
Hence, given a $\mathbf{U}^\imath$-module $M$ and a $\mathbf{U}$-module $N$, we can equip the tensor product $M \otimes N$ with a $\mathbf{U}^\imath$-module structure.

For each $\nu \in \mathrm{Par}_{\leq n}$, there exists a unique, up to isomorphism, simple weight $\mathbf{U}^\imath$-module $V^\imath(\nu)$ of highest weight $\nu$ (\cite[Theorem 6.3]{Mol06}; see also \cite[Proposition 3.3.9 and Corollary 4.3.2]{Wat21b}).

\begin{prop}[{\cite[Proposition 6.4.1]{Wat23f}}]\label{prop: biT}
  Let $\nu \in \mathrm{Par}_{\leq n}$.
  Then, there exists an almost orthonormal basis of $V^\imath(\nu)$ of the form
  \[
    \{ b^\imath_T \mid T \in Sp\mathrm{T}_{2n}(\nu) \}
  \]
  and a $\mathbf{U}^\imath$-module homomorphism $p_\nu : V(\nu) \to V^\imath(\nu)$ such that
  \[
    p_\nu(b_T) \equiv_\infty \begin{cases}
      b^\imath_T & \text{ if } T \in Sp\mathrm{T}_{2n}(\nu), \\
      0 & \text{ if } T \notin Sp\mathrm{T}_{2n}(\nu),
    \end{cases}
  \]
  for all $T \in \mathrm{SST}_{2n}(\nu)$.
\end{prop}

\begin{prop}\label{prop: forget}
  Let $\nu \in \mathrm{Par}_{\leq n}$.
  Then, there exists an injective $\mathbf{U}^\imath$-module homomorphism
  \[
    \mathrm{For} : V^\imath(\nu) \to V(\nu)
  \]
  such that
  \[
    \mathrm{For}(b^\imath_T) \equiv_\infty b_T \quad \text{ for all } T \in Sp\mathrm{T}_{2n}(\nu).
  \]
\end{prop}
\begin{proof}
  The $\mathbf{U}^\imath$-module $V(\nu)$ is semisimple since it is equipped with a nondegenerate contragredient bilinear form.
  Hence, the assertion follows from Proposition~\ref{prop: biT}.
\end{proof}

\subsection{Littlewood--Richardson maps}
In this subsection, we fix $\lambda \in \mathrm{Par}_{\leq 2n}$.

Given a word $w = (x_1,\dots,x_l) \in \mathcal{W}_{2n}$ and integers $a,b \in \mathbb{Z}$, set
\[
  w[a,b] := (w_a,w_{a+1},\dots,w_b), \quad w[a] := w[1,a].
\]

Let $T \in \mathrm{SST}_{2n}(\lambda)$.
Let $\mathbf{a} = (a_1,\dots,a_l)$ denote the first column of $T$ (read from the top to bottom), and $S$ the remaining part:
\[
  l := \ell(\lambda), \quad a_i := T(i,1), \quad S(i,j) := T(i,j+1).
\]
Define a subword $\operatorname{rem}(\mathbf{a}) \subseteq \mathbf{a}$ by the following recursive formula:
\[
  \operatorname{rem}(\mathbf{a}) := \begin{cases}
    \emptyset & \text{ if } l \leq 1, \\
    \mathrm{rem}(\mathbf{a}[l-2]) \cdot \mathbf{a}[l-1,l] & \text{ if } l \geq 2,\ a_l \in 2\mathbb{Z},\ a_{l-1} = a_l-1, \text{ and }\\
    & \phantom{if\ } a_l < 2l - |\mathrm{rem}(\mathbf{a}[l-2])| - 1, \\
    \mathrm{rem}(\mathbf{a}[l-1]) & \text{ otherwise}.
  \end{cases}
\]
Also, set $\operatorname{red}(\mathbf{a})$ to be the subword of $\mathbf{a}$ consisting of letters not in $\operatorname{rem}(\mathbf{a})$,
and $\operatorname{suc}(T)$ to be the unique semistandard tableau whose row-word is Knuth equivalent to $\operatorname{red}(\mathbf{a})^\mathrm{rev} \cdot w_\mathrm{col}(S)$, where $\mathrm{red}(\mathbf{a})^\mathrm{rev}$ denotes the reverse of $\mathrm{red}(\mathbf{a})$.

\begin{ex}\label{ex: suc}
  Let
  \[
    T := \begin{ytableau}
      1 & 1 & 2 & 2 & 3 \\
      3 & 3 & 3 & 8 \\
      4 & 6 & 8 \\
      6 & 8 & 9 \\
      8 \\
      9
    \end{ytableau}.
  \]
  Then,
  \[
    \mathbf{a} = (1,3,4,6,8,9),\ \operatorname{red}(\mathbf{a}) = (1,6,8,9),\ S = \begin{ytableau}
      1 & 2 & 2 & 3 \\
      3 & 3 & 8 \\
      6 & 8 \\
      8 & 9
    \end{ytableau}.
  \]
  The word $\operatorname{red}(\mathbf{a})^\mathrm{rev} \cdot w_\mathrm{col}(S)$ coincides with the column-word of $T$ in Example \ref{ex: rect}.
  Hence, by Corollary \ref{cor: rect and col-word}, we obtain
  \[
    \operatorname{suc}(T) =  \begin{ytableau}
      1 & 1 & 2 & 2 & 3 \\
      3 & 3 & 8 & 8 \\
      6 & 6 & 9 \\
      8 & 8 \\
      9
    \end{ytableau}.
  \]
\end{ex}

\begin{lem}[{\cite[Corollary 4.3.8]{Wat23f}}]\label{lem: symp suc}
  For each $T \in \mathrm{SST}_{2n}(\lambda)$, the following are equivalent.
  \begin{enumerate}
    \item $T$ is symplectic.
    \item $\mathrm{suc}(T) = T$.
  \end{enumerate}
\end{lem}

\begin{defi}[{\cite[\S 3.1]{Wat23f}}]\label{def: lr AII}
  The \emph{Littlewood--Richardson map} of type $A\mathrm{II}$ is the algorithm which takes $T \in \mathrm{SST}_{2n}(\lambda)$ as input, and outputs a pair $(P(T),Q(T))$ of tableaux as follows.
  \begin{enumerate}
    \item Initialize $t := 0$, $P^t := T$, $\nu^t := \operatorname{sh}(P^t)$, and $Q^t$ to be the unique skew tableau of shape $\lambda/\lambda$.
    \item\label{step: 2 def: lr AII} Set $P^{t+1} := \mathrm{suc}(P^t)$, $\nu^{t+1} := \operatorname{sh}(P^{t+1})$, and $Q^{t+1}$ to be the skew tableau of shape $\lambda/\nu^{t+1}$ defined by
    \[
      Q^{t+1}(i,j) := \begin{cases}
        Q^t(i,j) & \text{ if } (i,j) \notin D(\nu^t), \\
        t+1 & \text{ if } (i,j) \in D(\nu^t),
      \end{cases}
    \]
    for each $(i,j) \in D(\lambda/\nu^{t+1})$.
    \item \begin{enumerate}
      \item If $P^{t+1} = P^t$, then return $(P^t,Q^t)$.
      \item Otherwise, increment $t$, and then go back to step \eqref{step: 2 def: lr AII}.
    \end{enumerate}
  \end{enumerate}
\end{defi}

Set
\[
  \mathrm{Rec}_{2n}(\lambda/\nu) := \{ Q(T) \mid T \in \mathrm{SST}_{2n}(\lambda) \text{ such that } \operatorname{sh}(P(T)) = \nu \}.
\]
An explicit description of the set $\mathrm{Rec}_{2n}(\lambda/\nu)$ can be found in \cite[Theorem 3.1.4 (2)]{Wat23f}.

\begin{thm}[{\cite[Theorem 3.1.4 (1)]{Wat23f}}]
  The Littlewood--Richardson map of type $A\mathrm{II}$ gives rise to a bijection
  \[
    \mathrm{SST}_{2n}(\lambda) \to \bigsqcup_{\substack{\nu \in \mathrm{Par}_{\leq n} \\ \nu \subseteq \lambda}} (Sp\mathrm{T}_{2n}(\nu) \times \mathrm{Rec}_{2n}(\lambda/\nu)).
  \]
\end{thm}

\begin{thm}[{\cite[Theorem 7.2.1]{Wat23f}}]\label{thm: qua lr}
  There exists a $\mathbf{U}^\imath$-module isomorphism
  \[
    \mathrm{LR} : V(\lambda) \to \bigoplus_{\substack{\nu \in \mathrm{Par}_{\leq n} \\ \nu \subseteq \lambda}} (V^\imath(\nu) \otimes \mathbb{Q}(q) \mathrm{Rec}_{2n}(\lambda/\nu))
  \]
  such that
  \[
    \mathrm{LR}(b_T) \equiv_\infty b^\imath_{P(T)} \otimes Q(T) \quad \text{ for all } T \in \mathrm{SST}_{2n}(\lambda).
  \]
  Here, the linear space $\mathbb{Q}(q) \mathrm{Rec}_{2n}(\lambda/\nu)$ is regarded as a $\mathbf{U}$-module by the trivial action, and equipped with an almost orthonormal basis $\mathrm{Rec}_{2n}(\lambda,\nu)$.
\end{thm}

\subsection{Row-insertion of type $A\mathrm{II}$}
In this subsection, we fix $\nu \in \mathrm{Par}_{\leq n}$.

Based on the representation theoretical interpretation of the row-insertion (Theorem \ref{thm: sst and qg} \eqref{item: row ins, thm: sst and qg}), it is natural to define the row-insertion of type $A\mathrm{II}$ as follows.

\begin{defi}
  The \emph{row-insertion of type $A\mathrm{II}$} is the procedure which transforms a pair $(T,x) \in Sp\mathrm{T}_{2n}(\nu) \times [2n]$ into a new symplectic tableau $T \xleftarrow{A\mathrm{II}} x$ by
  \[
    T \xleftarrow{A\mathrm{II}} x := P(T \leftarrow x).
  \]
\end{defi}

\begin{ex}
  Let
  \[
    T := \begin{ytableau}
      1 & 1 & 2 & 3 & 3 \\
      3 & 3 & 4 & 8 \\
      6 & 6 & 8 \\
      8 & 8 & 9 \\
      9
    \end{ytableau}.
  \]
  Then,
  \[
    T \leftarrow 2 = \begin{ytableau}
      1 & 1 & 2 & *(lightgray) 2 & 3 \\
      3 & 3 & *(lightgray) 3 & 8 \\
      *(lightgray) 4 & 6 & 8 \\
      *(lightgray) 6 & 8 & 9 \\
      *(lightgray) 8 \\
      *(lightgray) 9
    \end{ytableau} \xrightarrow{P(\cdot)} \begin{ytableau}
      1 & 1 & 2 & 2 & 3 \\
      3 & 3 & 8 & 8 \\
      6 & 6 & 9 \\
      8 & 8 \\
      9
    \end{ytableau} = T \xleftarrow{A\mathrm{II}} 2.
  \]
\end{ex}

\begin{thm}\label{thm: row ins AII = ber row ins}
  The row-insertion of type $A\mathrm{II}$ coincides with the Berele row-insertion\rm{:}
  \[
    T \xleftarrow{A\mathrm{II}} x = T \xleftarrow{\mathrm{B}} x \quad \text{ for all } (T,x) \in Sp\mathrm{T}_{2n}(\nu) \times [2n].
  \]
\end{thm}
\begin{proof}
  Let $T \in Sp\mathrm{T}_{2n}(\nu)$ and $x \in [2n]$.
  Set $T' := T \leftarrow x$, $S := T \xleftarrow{\mathrm{B}} x$ and $U := T \xleftarrow{A\mathrm{II}} x$.
  Note that we have $U = P(T')$ by definition.
  Let $(c_1,\dots,c_r)$ and $(x_1,\dots,x_r)$ denote the row-insertion route and the row-inserting letters for $(T,x)$, respectively.

  First, suppose that $x_i \geq 2i-1$ for all $i \in [r]$.
  By Lemma \ref{lem: ber row ins route} \eqref{item 1 lem: ber row ins route}, we have
  \[
    S = T'.
  \]
  This implies that $T'$ is symplectic.
  Hence, Lemma \ref{lem: symp suc} implies that
  \[
    P(T') = T'.
  \]
  Therefore, the assertion follows.

  Next, suppose that $x_i < 2i-1$ for some $i \in [r]$.
  Let $s$ denote the minimum among those $i$.
  Let $S'$ denote the punctured tableau of shape $(\nu, \{(s-1,1)\})$ defined by
  \[
    S'(i,j) := \begin{cases}
      x_i & \text{ if } i \in [s-1] \text{ and } j = c_i, \\
      T(i,j) & \text{ otherwise},
    \end{cases}
  \]
  for each $(i,j) \in D(\nu,\{(s-1,1)\})$.
  By Lemma \ref{lem: ber row ins route} \eqref{item 2 lem: ber row ins route}, we have $S = \operatorname{Rect}(S')$.
  Then, Corollary \ref{cor: rect and col-word} implies that
  \begin{align}\label{eq: row S equiv col S'}
    w_\mathrm{row}(S) \equiv w_\mathrm{col}(S').
  \end{align}

  Now, let us compute $U$.
  Set $U' := \operatorname{suc}(T')$.
  Then, $w_\mathrm{row}(U') \equiv \operatorname{red}(\mathbf{a})^\mathrm{rev} \cdot w_\mathrm{col}(T'')$, where $\mathbf{a}$ denotes the first column of $T'$ and $T''$ the remaining part.
  By Lemma \ref{lem: symp-ness of row ins} \eqref{item 2 lem: symp-ness of row ins},we have
  \begin{itemize}
    \item $a_{s-1} = 2s-3$,
    \item $a_s = 2s-2$,
    \item $a_i \geq 2i-3$ for all $i > s$.
  \end{itemize}
  Hence, we see by the definition that
  \[
    \operatorname{rem}(\mathbf{a}) = (a_{s-1}, a_s).
  \]
  Therefore, $\operatorname{red}(\mathbf{a})^\mathrm{rev} \cdot w_\mathrm{col}(T'') = w_\mathrm{col}(S')$.
  By Proposition \ref{prop: row word equiv} and equation \eqref{eq: row S equiv col S'}, we obtain
  \[
    U' = S.
  \]
  This implies that $U'$ is symplectic, and hence, $U = U'$ by Lemma \ref{lem: symp suc}.
  Thus, we complete the proof.
\end{proof}

\begin{cor}
  Let $T \in Sp\mathrm{T}_{2n}(\nu)$ and $x \in [2n]$ such that $T' := T \leftarrow x$ is not symplectic.
  Let $s$ denote the minimum integer such that $T'(s,1) < 2s-1$, and $T''$ the punctured tableau obtained from $T'$ by replacing the $(s-1,1)$- and $(s,1)$-entry with holes.
  Then, we have
  \[
    T \xleftarrow{A\mathrm{II}} x = \operatorname{Rect}(T'').
  \]
\end{cor}
\begin{proof}
  The assertion follows from the proof of Theorem \ref{thm: row ins AII = ber row ins}.
\end{proof}

\section{Applications}\label{sect: app}
In this section, we use Theorem \ref{thm: row ins AII = ber row ins} to lift the Berele RS correspondence and Kobayashi--Matsumura RSK correspondence to isomorphisms of $\mathbf{U}^\imath$-modules.
Also, we establish the dual RSK correspondence of type $A\mathrm{II}$.

In this section, we fix $\nu \in \mathrm{Par}_{\leq n}$.

\subsection{Berele Robinson--Schensted correspondence}
\begin{prop}\label{prop: qua Ber row ins}
  There exists a $\mathbf{U}^\imath$-module isomorphism
  \[
    \mathrm{RI}^{A\mathrm{II}}: V^\imath(\nu) \otimes V(1) \to \bigoplus_{\substack{\xi \in \mathrm{Par}_{\leq n} \\ \xi \subset \nu, \ |\nu/\xi| = 1}} V^\imath(\xi) \oplus \bigoplus_{\substack{\xi \in \mathrm{Par}_{\leq n} \\ \xi \supset \nu, \ |\xi/\nu| = 1}} V^\imath(\xi)
  \]
  such that
  \[
    \mathrm{RI}^{A\mathrm{II}}(b^\imath_T \otimes b_x) \equiv_\infty b^\imath_{T \xleftarrow{A\mathrm{II}} x} \quad \text{ for all } (T,x) \in Sp\mathrm{T}_{2n}(\nu) \times [2n].
  \]
\end{prop}
\begin{proof}
  By Proposition \ref{prop: forget}, Theorem \ref{thm: sst and qg} \eqref{item: row ins, thm: sst and qg}, and Theorem \ref{thm: qua lr}, we can consider the following composition of $\mathbf{U}^\imath$-module homomorphisms:
  \begin{align*}
    V^\imath(\nu) \otimes V(1) &\xrightarrow{\mathrm{For} \otimes 1} V(\nu) \otimes V(1) \\
    &\xrightarrow{\mathrm{RI}} \bigoplus_{\substack{\mu \in \mathrm{Par}_{\leq 2n} \\ \mu \supset \nu,\ |\mu/\nu| = 1}} V(\mu) \\
    &\xrightarrow{\mathrm{LR}} \bigoplus_{\substack{\mu \in \mathrm{Par}_{\leq 2n} \\ \mu \supset \nu,\ |\mu/\nu| = 1}} \bigoplus_{\substack{\xi \in \mathrm{Par}_{\leq n} \\ \xi \subseteq \mu}} (V^\imath(\xi) \otimes \mathbb{Q}(q) \mathrm{Rec}_{2n}(\mu/\xi)) \\
    &\xrightarrow{\sum_\xi \pi_\xi} \bigoplus_{\xi \in \mathrm{Par}_{\leq n}} V^\imath(\xi),
  \end{align*}
  where the last homomorphism denotes the sum of projections $\pi_\xi$ onto $V(\xi)$.
  For each $T \in Sp\mathrm{T}_{2n}(\nu)$ and $x \in [2n]$, the element $b^\imath_T \otimes b_x$ is transformed by this homomorphism as follows modulo the equivalence relations $\equiv_\infty$ at each step:
  \[
    b^\imath_T \otimes b_x \mapsto b_T \otimes b_x \mapsto b_{T \leftarrow x} \mapsto b^\imath_{P(T \leftarrow x)} \otimes Q(T \leftarrow x) \mapsto b^\imath_{P(T \leftarrow x)} = b^\imath_{T \xleftarrow{A\mathrm{II}} x} = b^\imath_{T \xleftarrow{\mathrm{B}} x};
  \]
  the last equality follows from Theorem \ref{thm: row ins AII = ber row ins}.
  This assignment gives rise to the bijection in Theorem \ref{thm: ber row ins}.
  Hence, the assertion follows.
\end{proof}

\begin{defi}[{\cite[Definition 8.1]{Sun86}}; see also {\cite[Lemma 1]{Ber86}}]\label{def: ot}
  Let $\xi \in \mathrm{Par}_{\leq n}$ and $N \in \mathbb{Z}_{\geq 0}$.
  An \emph{oscillating tableau}, of shape $(\nu,\xi)$, rank $n$, and size $N$ is a sequence
  \[
    (\nu^0,\nu^1,\dots,\nu^N)
  \]
  of partitions in $\mathrm{Par}_{\leq n}$ satisfying the following:
  \begin{enumerate}
    \item $\nu^0 = \nu$, $\nu^N = \xi$,
    \item either $\nu^{i-1} \subset \nu^i$ or $\nu^{i-1} \supset \nu^i$ for all $i \in [N]$.
    \item $|\nu^i|-|\nu^{i-1}| = \pm 1$ for all $i \in [N]$,
  \end{enumerate}
  The set of oscillating tableaux of shape $(\nu,\xi)$, rank $n$, and size $N$ is denoted by $\mathrm{OT}_{n,N}(\nu,\xi)$.
\end{defi}

For each $T \in Sp\mathrm{T}_{2n}(\nu)$, $N \in \mathbb{Z}_{\geq 0}$, and $x_1,\dots,x_N \in [2n]$, set
\begin{align*}
  P(T,x_1,\dots,x_N) := \begin{cases}
    T & \text{ if } N = 0, \\
    P(T,x_1,\dots,x_{N-1}) \xleftarrow{A\mathrm{II}} x_N & \text{ if } N > 0,
  \end{cases}
\end{align*}
and $Q(T,x_1,\dots,x_N)$ to be the sequence $(\nu^0,\nu^1,\dots,\nu^N)$ of partitions defined by
\[
  \nu^i := \operatorname{sh}(P(T,x_1,\dots,x_i)) \quad \text{ for each } i \in [0,N].
\]

\begin{thm}\label{thm: rs aii}
  Let $N \in \mathbb{Z}_{\geq 0}$.
  \begin{enumerate}
    \item\label{item: comb, thm: rs aii} The assignment $(T,x_1,\dots,x_N) \mapsto (P(T,x_1,\dots,x_N),Q(T,x_1,\dots,x_N))$ gives rise to a bijection 
    \[
      \mathrm{RS}^\mathrm{B} : Sp\mathrm{T}_{2n}(\nu) \times [2n]^N \to \bigsqcup_{\xi \in \mathrm{Par}_{\leq n}} (Sp\mathrm{T}_{2n}(\xi) \times \mathrm{OT}_{n,N}(\nu,\xi)).
    \]
    \item\label{item: mod, thm: rs aii} There exists a $\mathbf{U}^\imath$-module isomorphism
    \[
      \mathrm{RS}^{A\mathrm{II}} : V^\imath(\nu) \otimes V(1)^{\otimes N} \to \bigoplus_{\xi \in \mathrm{Par}_{\leq n}} (V^\imath(\xi) \otimes \mathbb{Q}(q) \mathrm{OT}_{n,N}(\nu,\xi))
    \]
    such that
    \[
      \mathrm{RS}^{A\mathrm{II}}(b^\imath_T \otimes b_w) \equiv_\infty b^\imath_{P(T,x_1,\dots,x_N)} \otimes Q(T,x_1,\dots,x_N)
    \]
    for all $T \in Sp\mathrm{T}_{2n}(\nu)$, $x_1,\dots,x_N \in [2n]$.
  \end{enumerate}
\end{thm}
\begin{proof}
  The first assertion is essentially the Berele RS correspondence \cite[Theorem 2]{Ber86}.
  The second can be proved by induction on $N$ by using the first assertion and Proposition \ref{prop: qua Ber row ins}.
\end{proof}

\subsection{Kobayashi--Matsumura Robinson--Schensted--Knuth correspondence}
For each $l \in \mathbb{Z}_{\geq 0}$, let $\mathcal{W}_{2n}^\leq(l)$ denote the set of weakly increasing words of length $l$ with letters in $[2n]$.
We often identify $\mathrm{SST}_{2n}(l)$ with $\mathcal{W}_{2n}^\leq(l)$ via row-reading: $T \mapsto w_\mathrm{row}(T)$.

Let $l \in \mathbb{Z}_{\geq 0}$, $T \in Sp\mathrm{T}_{2n}(\nu)$, and $w = (x_1,\dots,x_l) \in \mathcal{W}_{2n}^\leq(l)$.
Set
\[
  P(T,w) := P(T,x_1,\dots,x_l).
\]
For each $i \in [l]$, let $r_i \in \mathcal{R}$ denote the terminal row of the Berele row-insertion for $(P(T,w[i-1]),x_i)$.
By Lemma \ref{lem: ber row ins} \eqref{item: row, lem: ber row ins}, it is weakly decreasing.
Hence, if $j$ denotes the number of barred letters in $(r_1,\dots,r_l)$ and if we set $\nu' := \operatorname{sh}(P(T,w[j]))$ and $\xi := \operatorname{sh}(P(T,w))$, then we have
\begin{align}\label{eq: hor osc cond}
  \nu \overset{\text{hor}}{\supseteq} \nu' \overset{\text{hor}}{\subseteq} \xi, \quad |\nu/\nu'| = j, \quad |\xi/\nu'| = l-j.
\end{align}
Set
\[
  Q(T,w) := (\nu,\nu',\xi).
\]

Conversely, given $\nu,\nu',\xi \in \mathrm{Par}_{\leq n}$ satisfying the condition~\eqref{eq: hor osc cond} and $S \in Sp\mathrm{T}_{2n}(\xi)$, one can find a unique pair $(T,w) \in Sp\mathrm{T}_{2n}(\nu) \times \mathcal{W}_{2n}^\leq(l)$ such that
\[
  P(T,w) = S, \quad Q(T,w) = (\nu,\nu',\xi).
\]

\begin{prop}\label{prop: pieri aii}
  Let $l \in \mathbb{Z}_{\geq 0}$.
  \begin{enumerate}
    \item The assignment $(T,w) \mapsto (P(T,w),Q(T,w))$ gives rise to a bijection
    \[
      Sp\mathrm{T}_{2n}(\nu) \times \mathcal{W}_{2n}^\leq(l) \to \bigsqcup_{\substack{\nu',\xi \in \mathrm{Par}_{\leq n} \\ \nu \overset{\text{hor}}{\supseteq} \nu' \overset{\text{hor}}{\subseteq} \xi \\ |\nu/\nu'| + |\xi/\nu'| = l}} (Sp\mathrm{T}_{2n}(\xi) \times \{ (\nu,\nu',\xi) \}).
    \]
    \item There exists a $\mathbf{U}^\imath$-module isomorphism
    \[
      V^\imath(\nu) \otimes V(l) \to \bigoplus _{\substack{\nu',\xi \in \mathrm{Par}_{\leq n} \\ \nu \overset{\text{hor}}{\supseteq} \nu' \overset{\text{hor}}{\subseteq} \xi \\ |\nu/\nu'| + |\xi/\nu'| = l}} (V^\imath(\xi) \otimes \mathbb{Q}(q) \{ (\nu,\nu',\xi) \})
    \]
    that sends $b^\imath_T \otimes b_w$ to $b^\imath_{P(T,w)} \otimes Q(T,w)$ modulo $\equiv_\infty$.
  \end{enumerate}
\end{prop}
\begin{proof}
  The first assertion follows from the argument above.
  In order to prove the second assertion, let us consider the following composition of $\mathbf{U}^\imath$-module homomorphisms in Theorems \ref{thm: sst and qg} \eqref{item: row word, thm: sst and qg} and \ref{thm: rs aii} \eqref{item: mod, thm: rs aii}:
  \begin{align*}
    V^\imath(\nu) \otimes V(l)
    &\xrightarrow{1 \otimes \mathrm{RR}} V^\imath(\nu) \otimes V(1)^{\otimes l} \\
    &\xrightarrow{\mathrm{RS}^{A\mathrm{II}}} \bigoplus_{\xi \in \mathrm{Par}_{\leq n}} (V^\imath(\xi) \otimes \mathbb{Q}(q) \mathrm{OT}_{n,l}(\nu,\xi)).
  \end{align*}
  For each $T \in Sp\mathrm{T}_{2n}(\nu)$ and $w = (x_1,\dots,x_l) \in \mathcal{W}_{2n}^\leq(l)$, the element $b^\imath_T \otimes b_w$ is transformed by this homomorphism as follows modulo the equivalence relations $\equiv_\infty$ at each step:
  \[
    b^\imath_T \otimes b_w \mapsto b^\imath_T \otimes b_{x_1} \otimes \cdots \otimes b_{x_l} \mapsto b^\imath_{P(T,x_1,\dots,x_l)} \otimes Q(T,x_1,\dots,x_l).
  \]
  Note that $P(T,x_1,\dots,x_l) = P(T,w)$ and that $Q(T,x_1,\dots,x_l)$ is of the form $(\nu^0,\nu^1,\dots,\nu^l)$ such that
  \[
    \nu^0 \overset{\text{hor}}{\supseteq} \nu^j \overset{\text{hor}}{\subseteq} \nu^l \quad \text{ for some } j \in [0,l].
  \]
  The triple $(\nu,\nu^j,\nu^l)$ is nothing but $Q(T,w)$.
  Hence, we obtain a $\mathbf{U}^\imath$-module homomorphism
  \[
    V^\imath(\nu) \otimes V(l) \to \bigoplus _{\substack{\nu',\xi \in \mathrm{Par}_{\leq n} \\ \nu \overset{\text{hor}}{\supseteq} \nu' \overset{\text{hor}}{\subseteq} \xi \\ |\nu/\nu'| + |\xi/\nu'| = l}} (V^\imath(\xi) \otimes \mathbb{Q}(q) \{ (\nu,\nu',\xi) \})
  \]
  which sends $b^\imath_T \otimes b_w$ to $b^\imath_{P(T,w)} \otimes Q(T,w)$ modulo $\equiv_\infty$.
  The first assertion ensures that this map is an isomorphism.
  Thus, we complete the proof.
\end{proof}

\begin{defi}[{\cite[Definition 4.7]{KoMa25}}]\label{def: csot}
  Let $\xi \in \mathrm{Par}_{\leq n}$.
  A \emph{column-strict oscillating tableau} of shape $(\nu,\xi)$, rank $n$, and depth $k$ is a sequence
  \[
    (\xi^0,\nu^1,\xi^1,\dots,\nu^k,\xi^k)
  \]
  of partitions in $\mathrm{Par}_{\leq n}$ satisfying the following:
  \begin{enumerate}
    \item $\xi^0 = \nu$, $\xi^k = \xi$,
    \item $\xi^{i-1} \overset{\text{hor}}{\supseteq} \nu^i \overset{\text{hor}}{\subseteq} \xi^i$ for all $i \in [k]$.
  \end{enumerate}
  Let $\mathrm{CSOT}_{n,k}(\nu,\xi)$ denote the set of column-strict oscillating tableaux of shape $(\nu,\xi)$, rank $n$, and depth $k$.
\end{defi}

\begin{rem}\label{rem: ssot}
  The notion of column-strict oscillating tableaux was introduced in \cite{KoMa25} under the name of \emph{semistandard oscillating tableaux}.
  As stated there, Lee \cite{Lee25b} had introduced the notion of semistandard oscillating tableaux in a different meaning.
  Hence, we give them different names which clarify their different roles in the present paper.
\end{rem}

\begin{defi}
  The content of $U \in \mathrm{CSOT}_{n,k}(\nu,\xi)$ is the sequence $c(U) := (l_1,\dots,l_k)$ defined as follows:
  if $U = (\xi^0,\nu^1,\xi^1,\dots,\nu^k,\xi^k)$, then
  \[
    l_i := |\xi^{i-1}/\nu^i| + |\xi^i/\nu^i| \quad \text{ for each } i \in [k].
  \]
\end{defi}

For each $T \in Sp\mathrm{T}_{2n}(\nu)$, $k \in \mathbb{Z}_{\geq 0}$, $l_1,\dots,l_k \in \mathbb{Z}_{\geq 0}$, and $w_i \in \mathcal{W}_{2n}^\leq(l_i)$ with $i \in [k]$, set
\[
  P(T,w_1,\dots,w_k) := \begin{cases}
    T & \text{ if } k = 0, \\
    P(P(T,w_1,\dots,w_{k-1}),w_k) & \text{ if } k > 0,
  \end{cases}
\]
and $Q(T,w_1,\dots,w_k)$ to be the sequence $(\xi^0,\nu^1,\xi^1,\dots,\nu^k,\xi^k)$ of partitions defined as follows:
\begin{itemize}
  \item $\xi^0 := \nu$,
  \item $(\xi^{i-1},\nu^i,\xi^i) = Q(P(T,w_1,\dots,w_{i-1}),w_i)$ for each $i \in [k]$.
\end{itemize}

\begin{thm}\label{thm: rsk aii}
  Let $k \in \mathbb{Z}_{\geq 0}$ and $l_1,\dots,l_k \in \mathbb{Z}_{\geq 0}$.
  \begin{enumerate}
    \item The assignment $(T,w_1,\dots,w_k) \mapsto (P(T,w_1,\dots,w_k),Q(T,w_1,\dots,w_k))$ gives rise to a bijection
    \begin{align*}
      \mathrm{RSK}^\mathrm{KM} : &Sp\mathrm{T}_{2n}(\nu) \times \mathcal{W}_{2n}^{\leq}(l_1) \times \cdots \times \mathcal{W}_{2n}^\leq(l_k) \\
      &\to \bigsqcup_{\xi \in \mathrm{Par}_{\leq n}} (Sp\mathrm{T}_{2n}(\xi) \times \{ U \in \mathrm{CSOT}_{n,k}(\nu,\xi) \mid c(U) = (l_1,\dots,l_k) \}).
    \end{align*}
    \item There exists a $\mathbf{U}^\imath$-module isomorphism
    \begin{align*}
      \mathrm{RSK}^{A\mathrm{II}} : &V^\imath(\nu) \otimes V(l_1) \otimes \cdots \otimes V(l_k) \\
      &\to \bigoplus_{\xi \in \mathrm{Par}_{\leq n}} V^\imath(\xi) \otimes \mathbb{Q}(q) \{ U \in \mathrm{CSOT}_{n,k}(\nu,\xi) \mid c(U) = (l_1,\dots,l_k) \}
    \end{align*}
    such that
    \[
      \mathrm{RSK}^{A\mathrm{II}}(b^\imath_T \otimes b_{w_1} \otimes \cdots \otimes b_{w_k}) \equiv_\infty b^\imath_{P(T,w_1,\dots,w_k)} \otimes Q(T,w_1,\dots,w_k)
    \]
    for all $T \in Sp\mathrm{T}_{2n}(\nu)$, $w_i \in \mathcal{W}_{2n}^\leq(l_i)$.
  \end{enumerate}
\end{thm}

\begin{proof}
  The first assertion is essentially the Kobayashi--Matsumura RSK-correspondence \cite[Theorem 4.17]{KoMa25}.
  The second can be proved by induction on $k$ and by using the first assertion and Proposition \ref{prop: pieri aii}.
\end{proof}

\subsection{Dual Robinson--Schensted--Knuth correspondence of type $A\mathrm{II}$}
For each $k \in [2n]$, let $(1^k)$ denote the partition consisting of $k$ $1$'s:
\[
  (1^k) := (\underbrace{1,\dots,1}_k).
\]
Also, let $\mathcal{W}_{2n}^>(k)$ denote the set of strictly decreasing words of length $k$ with letters in $[2n]$.
We often identify $\mathrm{SST}_{2n}(1^k)$ with $\mathcal{W}_{2n}^>(k)$ via row-reading: $T \mapsto w_\mathrm{row}(T)$.

Let $k \in [0,2n]$, $T \in Sp\mathrm{T}_{2n}(\nu)$, and $w = (x_1,\dots,x_k) \in \mathcal{W}_{2n}^>(k)$.
Set
\[
  P(T,w) := P(T,x_1,\dots,x_k).
\]
For each $i \in [k]$, let $r_i \in \mathcal{R}$ denote the terminal row of the Berele row-insertion for $(P(T,w[i-1]),x_i)$.
By Lemma \ref{lem: ber row ins} \eqref{item: col, lem: ber row ins}, it is strictly increasing.
Hence, if $j$ denotes the number of unbarred letters in $(r_1,\dots,r_k)$ and if we set $\nu' := \operatorname{sh}(P(T,w[j]))$ and $\xi := \operatorname{sh}(P(T,w))$, then we have
\begin{align}\label{eq: ver osc cond}
  \nu \overset{\text{ver}}{\subseteq} \nu' \overset{\text{ver}}{\supseteq} \xi, \quad |\nu'/\nu| = j, \quad |\nu'/\xi| = k-j.
\end{align}
Set
\[
  Q(T,w) := (\nu,\nu',\xi).
\]

Conversely, given partitions $\nu,\nu',\xi \in \mathrm{Par}_{\leq n}$ satisfying the condition~\eqref{eq: ver osc cond} and $S \in Sp\mathrm{T}_{2n}(\xi)$, one can find a unique pair $(T,w) \in Sp\mathrm{T}_{2n}(\nu) \times \mathcal{W}_{2n}^>(k)$ such that
\[
  P(T,w) = S, \quad Q(T,w) = (\nu,\nu',\xi).
\]

\begin{prop}\label{prop: dual pieri aii}
  Let $k \in [0,2n]$.
  \begin{enumerate}
    \item The assignment $(T,w) \mapsto (P(T,w),Q(T,w))$ gives rise to a bijection
    \[
      Sp\mathrm{T}_{2n}(\nu) \times \mathcal{W}_{2n}^>(k) \to \bigsqcup_{\substack{\nu',\xi \in \mathrm{Par}_{\leq n} \\ \nu \overset{\text{ver}}{\subseteq} \nu' \overset{\text{ver}}{\supseteq} \xi \\ |\nu'/\nu| + |\nu'/\xi| = k}} (Sp\mathrm{T}_{2n}(\xi) \times \{ (\nu,\nu',\xi) \}).
    \]
    \item There exists a $\mathbf{U}^\imath$-module isomorphism
    \[
      V^\imath(\nu) \otimes V(k) \to \bigoplus _{\substack{\nu',\xi \in \mathrm{Par}_{\leq n} \\ \nu \overset{\text{ver}}{\subseteq} \nu' \overset{\text{ver}}{\supseteq} \xi \\ |\nu'/\nu| + |\nu'/\xi| = k}} (V^\imath(\xi) \otimes \mathbb{Q}(q) \{ (\nu,\nu',\xi) \})
    \]
    that sends $b^\imath_T \otimes b_w$ to $b^\imath_{P(T,w)} \otimes Q(T,w)$ modulo $\equiv_\infty$.
  \end{enumerate}
\end{prop}
\begin{proof}
  The assertion can be proved in a similar way to Proposition \ref{prop: pieri aii}.
\end{proof}

\begin{defi}[{\cite[Definition 2.4]{Lee25b}}]\label{def: rsot}
  Let $\xi \in \mathrm{Par}_{\leq n}$.
  A \emph{row-strict oscillating tableau} of shape $(\nu,\xi)$, rank $n$, and width $l$ is a sequence
  \[
    (\xi^0,\nu^1,\xi^1,\dots,\nu^l,\xi^l)
  \]
  of partitions in $\mathrm{Par}_{\leq n}$ satisfying the following:
  \begin{enumerate}
    \item $\xi^0 = \nu$, $\xi^l = \xi$,
    \item $\xi^{i-1} \overset{\text{ver}}{\subseteq} \nu^i \overset{\text{ver}}{\supseteq} \xi^i$ for all $i \in [l]$.
  \end{enumerate}
  Let $\mathrm{RSOT}_{n,l}(\nu,\xi)$ denote the set of row-strict oscillating tableaux of shape $(\nu,\xi)$, rank $n$, and width $l$.
\end{defi}

\begin{rem}
  The notion of row-strict oscillating tableaux was introduced in \cite{Lee25b} under the name of \emph{semistandard oscillating tableaux}; see also Remark \ref{rem: ssot}.
\end{rem}

\begin{defi}
  The content of $U \in \mathrm{RSOT}_{n,l}(\nu,\xi)$ is the sequence $c(U) := (k_1,\dots,k_l)$ defined as follows:
  if $U = (\xi^0,\nu^1,\xi^1,\dots,\nu^l,\xi^l)$, then
  \[
    k_i := |\nu^i/\xi^{i-1}| + |\nu^i/\xi^i| \quad \text{ for each } i \in [l].
  \]
\end{defi}

For each $T \in Sp\mathrm{T}_{2n}(\nu)$, $l \in \mathbb{Z}_{\geq 0}$, $k_1,\dots,k_l \in [0,2n]$, and $w_i \in \mathcal{W}_{2n}^>(k_i)$ with $i \in [l]$, set
\[
  P(T,w_1,\dots,w_l) := \begin{cases}
    T & \text{ if } l = 0, \\
    P(P(T,w_1,\dots,w_{l-1}),w_l) & \text{ if } l > 0,
  \end{cases}
\]
and $Q(T,w_1,\dots,w_l)$ to be the sequence $(\xi^0,\nu^1,\xi^1,\dots,\nu^l,\xi^l)$ of partitions defined as follows:
\begin{itemize}
  \item $\xi^0 := \nu$,
  \item $(\xi^{i-1},\nu^i,\xi^i) = Q(P(T,w_1,\dots,w_{i-1}),w_i)$ for each $i \in [l]$.
\end{itemize}

\begin{thm}\label{thm: dual rsk aii}
  Let $l \in \mathbb{Z}_{\geq 0}$, and $k_1,\dots,k_l \in [0,2n]$.
  \begin{enumerate}
    \item\label{item: comb, thm: dual rsk aii} The assignment $(T,w_1,\dots,w_l) \mapsto (P(T,w_1,\dots,w_l),Q(T,w_1,\dots,w_l))$ gives rise to a bijection
    \begin{align*}
      \mathrm{dRSK}^{A\mathrm{II}}: &Sp\mathrm{T}_{2n}(\nu) \times \mathcal{W}_{2n}^>(k_1) \times \cdots \times \mathcal{W}_{2n}^>(k_l) \\
      &\to \bigsqcup_{\xi \in \mathrm{Par}_{\leq n}} (Sp\mathrm{T}_{2n}(\xi) \times \{ U \in \mathrm{RSOT}_{n,l}(\nu,\xi) \mid c(U) = (k_1,\dots,k_l) \}).
    \end{align*}
    \item\label{item: mod, thm: dual rsk aii} There exists a $\mathbf{U}^\imath$-module isomorphism
    \begin{align*}
      \mathrm{dRSK}^{A\mathrm{II}}: &V^\imath(\nu) \otimes V(1^{k_1}) \otimes \cdots \otimes V(1^{k_l}) \\
      &\to \bigoplus_{\xi \in \mathrm{Par}_{\leq n}} V^\imath(\xi) \otimes \mathbb{Q}(q) \{ U \in \mathrm{RSOT}_{n,l}(\nu,\xi) \mid c(U) = (k_1,\dots,k_l) \}
    \end{align*}
    such that
    \[
      \mathrm{dRSK}^{A\mathrm{II}}(b^\imath_T \otimes b_{w_1} \otimes \cdots \otimes b_{w_l}) \equiv_\infty b^\imath_{P(T,w_1,\dots,w_l)} \otimes Q(T,w_1,\dots,w_l)
    \]
    for all $T \in Sp\mathrm{T}_{2n}(\nu)$, $w_i \in \mathcal{W}_{2n}^>(k_i)$ with $i \in [l]$.
  \end{enumerate}
\end{thm}

\begin{proof}
  The assertion can be proved by induction on $k$ and by means of Proposition \ref{prop: dual pieri aii}.
\end{proof}

\end{document}